\documentclass[a4paper,12pt]{article}
\usepackage{amsmath}
\usepackage{amssymb}
\usepackage{amsthm}
\usepackage{amsfonts}
\usepackage{mathrsfs}
\usepackage{graphicx}
\usepackage{hyperref}
\usepackage{float}
\usepackage{slashed}
\usepackage{cite}
\usepackage{xcolor}
\usepackage{authblk}

\newtheorem{theorem}{Theorem}
\newtheorem{lemma}{Lemma}

\newtheorem{definition}{Definition}

\newtheorem{preposition}{Preposition}

\numberwithin{equation}{section}

\newtheoremstyle{named}{}{}{\itshape}{}{\bfseries}{.}{.5em}{\thmnote{#3 }#1}
\theoremstyle{named}

\hypersetup{ pdftitle={Report 2},
	pdfauthor={Mulyanto}, 
	pdfkeywords={black hole; decay estimate; Maxwell-Higgs; Morawetz estimate; RN}, 
	bookmarksnumbered, pdfstartview={FitH}, urlcolor=blue, }

\begin{document}
	\pagestyle{plain}

\title{\LARGE\textbf{Uniform Energy Bound for Maxwell-Higgs Equations on Reissner-Nordström Spacetimes}}

\author[1]{Mulyanto}
\author[1]{Ardian N. Atmaja}
\author[2]{Fiki T. Akbar}
\author[2]{Bobby E. Gunara}

\affil[1]{\textit{\small Research Center for Quantum Physics, National Research and Innovation Agency (BRIN),
Kompleks PUSPIPTEK Serpong, Tangerang 15310, Indonesia}}
\affil[2]{\textit{\small Theoretical Physics Laboratory,
	Theoretical High Energy Physics Research Division,
	Faculty of Mathematics and Natural Sciences,
	Institut Teknologi Bandung
	Jl. Ganesha no. 10 Bandung, Indonesia, 40132}}

\date{\small email: muly031@brin.go.id, ardian\textunderscore n\textunderscore a@yahoo.com, ftakbar@itb.ac.id, bobby@itb.ac.id}

\maketitle

\begin{abstract}
In this paper, we prove the uniform energy bound for the Maxwell-Higgs system in the exterior region of Reissner-Nordstr\"om black holes. By employing an integrated local energy decay (ILED) estimate in combination with the Sobolev embedding theorem on a compact Riemannian manifold, we derived $L^\infty$ bounds for the fields. These results were then used to obtain a bound for the conformal energy of the system using the Cauchy-Schwarz inequality and Hardy-type inequalities.
\end{abstract}

\section{Introduction}
It is of interest to consider the energy bound of the following Maxwell-Higgs (MH) system on  four dimensional Reissner-Nordstr\"om (RN) spacetimes. We begin from the Lagrangian  of MH system which can be written down as	
\begin{equation}\label{lagrangian}
\mathcal{L}=-\frac{1}{4}{{F}_{\mu \nu }}{{F}^{\mu \nu }}-{{D}_{\mu }}\phi \overline{{{D}^{\mu }}\phi }~,
\end{equation}
where $F_{\mu \nu } \equiv {{\nabla }_{\mu }}A_{\nu }-{{\nabla }_{\nu }}A_{\mu }$ is the gauge field strength and $D_{\mu}{\phi}\equiv\nabla _{\mu }{\phi }-iA_{\mu}{\phi }$ is the covariant derivative of the complex scalar fields. Lagrangian \eqref{lagrangian} describes the interaction of a two-form function $F_{\mu\nu}$ and a complex scalar field $\phi$ on 4-dimensional Reissner-Nordstr\"om  spacetimes with the standard coordinate $x^{\mu} = (t,x^i)$ where $\mu=0,1,2,3$ and $i=1,2,3$, endowed with metric,
\begin{eqnarray}\label{RNmetric}
g_{\mathcal{M}}=f(r)dt^2-\frac{1}{f(r)}dr^2-r^2d\omega^2,
\end{eqnarray}
where $d\omega^2$ is Euclidean metric on 2-sphere which in spherical coordinates is
\begin{eqnarray}
d\omega^2=d\theta^2+ \text{sin}^2 \theta d\varphi^2,
\end{eqnarray}
and
\begin{eqnarray}
f(r)=1-\frac{2M}{r}+\frac{Q^2}{r^2}.
\end{eqnarray}
The field equations of motion and the energy-momentum tensor for the MH system with respect to the Lagrangian \eqref{lagrangian} are 
\begin{eqnarray}\label{eom1}
{{\nabla }^{\alpha }}{{F}_{\alpha \gamma }}=-i\left( {{D}_{\gamma }}\phi \overline{\phi }-\phi \overline{{{D}_{\gamma }}\phi } \right)~,
\end{eqnarray}
\begin{eqnarray}\label{eom2}
{{D}^{\mu }}{{D}_{\mu }}\phi =0~,
\end{eqnarray}
\begin{equation}\label{TEM}
{{{T}}^{\mu \nu }}=F_{\gamma }^{\nu }{{F}^{\mu \gamma }}-\frac{1}{4}{{g}^{\mu \nu }}{{F}^{\alpha \beta }}{{F}_{\alpha \beta }}+{{D}^{\mu }}\phi \overline{{{D}^{\nu }}\phi }-\frac{1}{2}~{{g}^{\mu \nu }}~{{D}_{\gamma }}\phi ~\overline{{{D}^{\gamma }}\phi }.
\end{equation}
where 
\begin{eqnarray}
\nabla^{\alpha}F_{\alpha\gamma}=g^{\alpha\beta}\left(\partial_{\beta}F_{\alpha\gamma}-F_{\sigma\gamma}\Gamma^{\sigma}_{\beta\alpha}-F_{\alpha\sigma}\Gamma^{\sigma}_{\beta\gamma}\right)~.
\end{eqnarray}

The study of energy decay originated from the seminal work of C. S. Morawetz, who discovered how to use vector fields to derive local decay rates for wave equations on flat spacetimes \cite{morawetz1961, morawetz}. S. Klainerman then used the vector field method to prove the global existence of nonlinear wave equations in Minkowski spacetimes \cite{klainerman1980}. Subsequently, the groundbreaking concept of a generalized energy norm derived from a conformal group of vector fields was introduced by the same author. This provides a new framework for analyzing the decay rates of the wave equation in $\mathbb{R}^{n+1}$ \cite{klainerman1}. Klainerman’s approach essentially merged the local energy decay estimates established by C. S. Morawetz \cite{morawetz1962} with the conformal method developed by  Y. Choquet-Bruhat and D. Christodoulou \cite{cho}. The techniques were subsequently applied to yield numerous results on the long-term behavior of solutions and the global existence of fields on the Minkowski spacetimes  \cite{klainerman1987, gini, klainerman1994}.

The vector field method was later extended for application to curved spacetime as well. In general, even in compact manifolds, the energy of the solutions in a curved spacetime is not always bounded \cite{bruhat}. The first pioneering work on curved spacetimes was conducted by S. Klainerman \cite{klainerman1981}, who obtained the exact solution of the homogeneous  wave equation on Friedmann-Robertson-Walker spacetimes. Later, the vector field method was widely used to determine the decay estimate for various curved spacetimes \cite{blue1, blue, mokdad, Finster, Ander}. For the scalar field in the curved spacetimes case, the Killing energy degenerates in the horizon. Because $\partial_t$ is actually null at the event horizon of spacetimes, the energy associated with $\partial_t$  degenerates in derivatives transverse to the horizon, and it is not positive for the entire domain of exterior communication. These problems were resolved by M. Dafermos and I. Rodnianski by exploiting the redshift effect on Schwarzschild black holes \cite{dafermos}. For the Maxwell field cases, the conformal energy of the full Maxwell field is controlled by the energy of the middle component. To obtain the decay estimate, an estimate of the middle-component field is crucial. In \cite{ghanem}, this problem was somewhat resolved by assuming that the energy of the middle component is bounded by its reduced system energy. Later, M. Mokdad proved the energy estimate of middle component problems for Maxwell fields using tetrad formalism and integrated local decay energy estimation methods on Reissner Nordstr\"om de Sitter spacetimes \cite{mokdad}. However, relatively few studies have focused on the Maxwell-Higgs equation, especially in curved spacetimes, such as those of RN black holes. The challenge in determining the energy decay of the Maxwell-Higgs equation is to estimate the interaction term due to the two fields. The shape of the resulting wave-like equation became nonlinear. Thus, we can not employ directly the method on \cite{dafermos} or \cite{mokdad}. 

In this paper, we focus on the energy boundedness for the Maxwell-Higgs system on RN spacetimes. First, we cast the Maxwell-Higgs equations \eqref{eom1} into tetrad formalism, enabling them to take a nonlinear wave-like equation. Then, we define the system energy and conformal energy derived from the Morawetz vector field. To establish a bound for conformal energy, we must first determine the $L^\infty$ bound in the local region for the scalar and electromagnetic fields. This was achieved by applying the Sobolev embedding theorem to a compact Riemannian manifold. Furthermore, using this result, the bound for the source term (the interaction term between the scalar field and the electromagnetic field) is obtained. Finally, using the Cauchy-Schwarz inequality together with the Hardy-like inequality, we obtain the uniform bound for the conformal energy for a finite time.

This paper is organized as follows: In Section \ref{sec:MH}, we review the Maxwell-Higgs system in tetrad formalism to obtain the wavelike equation; in Section \ref{sec:estimate} we obtain the bound on $L^{\infty}_{loc}$ of the system; and in Section \ref{sec:main}, we prove the main results of this work, which is the uniform bound  for the Maxwell-Higgs equation on RN spacetimes.

\section{The Wave-like Equation}
\label{sec:MH}
In this section, we rewrite the Maxwell-Higgs equation in the tetrad formalism. As mentioned in the introduction, to control the energies of the Maxwell ﬁeld, we must obtain a bound on the middle component that satisfies a wave-like equation. We demonstrate that this can be expressed in terms of a wave-like equation. To express Maxwell-Higgs equations in a way that depends more on the geometry of the spacetime, we use a tetrad formalism as in \cite{mokdad}. 
\begin{definition}\label{def1}
Let the following tetrad be defined as:
\begin{eqnarray}
\label{L} L&\equiv&\partial_{t}+\partial_{r*}\\
\label{N} N&\equiv&\partial_{t}-\partial_{r*}\\
\label{M} M&\equiv&\partial_\theta+\frac{i}{\text{sin}\theta}\partial_{\varphi}\\
\label{barM} \bar{M}&\equiv&\partial_\theta-\frac{i}{\text{sin}\theta}\partial_{\varphi}
\end{eqnarray}
\end{definition}
The Maxwell field component could be represented by three complex scalar functions, known as the spin components of the Maxwell field, which are defined via the use of this tetrad.
\begin{eqnarray}
\Phi_0&=&\frac{1}{2}\left(V^{-1}F(L,M)+F(\bar{M},M)\right)~,\\
\Phi_1&=&F(L,M)~,\\
\Phi_{-1}&=&F(N,\bar{M})~.
\end{eqnarray}
where $V=\frac{f}{r^2}$.
Moreover, in this framework, the Maxwell-Higgs equations can also be translate as follows 
\begin{lemma}\label{maxwellcompact}
	$F,\phi$ satisﬁes Maxwell-Higgs equations \eqref{eom1} and \eqref{eom2} if and only if its spin components $(\Phi_{-1},\Phi_{0}, \Phi_{1})$ in the stationary tetrad satisfy the compacted equations
	\begin{eqnarray}
	\label{maxcom1}N\Phi_1&=&VM\Phi_{0}+if\left( {{D}_{2}}\phi \overline{\phi }-\phi \overline{{{D}_{2}}\phi } \right)+i\frac{f}{\text{sin}\theta}\left( {{D}_{3}}\phi \overline{\phi }-\phi \overline{{{D}_{3}}\phi } \right)~,\\
	\label{maxcom2}L\Phi_{0}&=&\bar{M}\Phi_{1}+ir^2\left( {{D}_{0}}\phi \overline{\phi }-\phi \overline{{{D}_{0}}\phi } \right)+ir^2\left( {{D}_{1}}\phi \overline{\phi }-\phi \overline{{{D}_{1}}\phi } \right)~,\\
	\label{maxcom3}N\Phi_{0}&=&-M_{1}\Phi_{-1}+ir^2\left( {{D}_{0}}\phi \overline{\phi }-\phi \overline{{{D}_{0}}\phi } \right)+ir^2\left( {{D}_{1}}\phi \overline{\phi }-\phi \overline{{{D}_{1}}\phi } \right)~,\\
	\label{maxcom4}L\Phi_{-1}&=&V\bar{M}\Phi_{0}+if\left( {{D}_{2}}\phi \overline{\phi }-\phi \overline{{{D}_{2}}\phi } \right)+i\frac{f}{\text{sin}\theta}\left( {{D}_{3}}\phi \overline{\phi }-\phi \overline{{{D}_{3}}\phi } \right)~,
	\end{eqnarray}
 where $M_1=M+cot\theta$ and $\bar{M}$ is its conjugate.
\end{lemma}
\begin{proof}
	Let express $(\Phi_{-1},\Phi_{0}, \Phi_{1})$ in terms of the component of $F$ in the coordinates $(t,r*,\theta,\varphi)$. We have
	\begin{eqnarray}
	F(L,M)&=&F_{02}+\frac{i}{\text{sin}\theta}F_{03}+F_{12}++\frac{i}{\text{sin}\theta}F_{13}~,\\
	F(L,N)&=&2F_{10}~,\\
	F(\bar{M},M)&=&\frac{2i}{\text{sin}\theta}F_{23}~,\\
	F(N,\bar{M})&=&F_{02}+\frac{i}{\text{sin}\theta}F_{30}+F_{21}++\frac{i}{\text{sin}\theta}F_{13}~.
	\end{eqnarray}
From definition \ref{def1}, we obtain
	\begin{eqnarray}
	\Phi_{-1}&=&F_{02}+\frac{i}{\text{sin}\theta}F_{30}+F_{21}++\frac{i}{\text{sin}\theta}F_{13}~,\\
	\Phi_{0}&=&\frac{r^2}{1-2m/r}F_{10}+\frac{i}{\text{sin}\theta}F_{23}~,\\
	\Phi_{1}&=&F_{02}+\frac{i}{\text{sin}\theta}F_{03}+F_{12}++\frac{i}{\text{sin}\theta}F_{13}~.
	\end{eqnarray}
Thus, we arrive at
	\begin{eqnarray}\label{com1}
	N\Phi_{1}&=&\left(\partial_{t}-\partial_{r*}\right)\left(F_{02}+\frac{i}{\text{sin}\theta}F_{03}+F_{12}++\frac{i}{\text{sin}\theta}F_{13}\right)\notag\\
	&=&\partial_0 F_{02}+\frac{i}{\text{sin}\theta}\partial_0 F_{03}+\partial_0 F_{12}+\frac{i}{\text{sin}\theta}\partial_0 F_{13}\notag\\
	&&+\partial_1 F_{20}+\partial_1 F_{21}+\frac{i}{\text{sin}\theta}\partial_1 F_{30}+\frac{i}{\text{sin}\theta}\partial_1 F_{31}~,
	\end{eqnarray}
	\begin{eqnarray}\label{com2}
	VM\Phi_{0}&=&V\left(\partial_\theta+\frac{i}{\text{sin}\theta}\partial_\varphi\right)\left(V^{-1}F_{10}+\frac{i}{\text{sin}\theta}F_{23}\right)\notag\\
	&=&\partial_2 F_{10}+\frac{i}{\text{sin}\theta}\partial_3 F_{10}+\frac{iV}{ \text{sin}\theta}F_{32}cot\theta\notag\\
	&&+\frac{iV}{\text{sin}\theta}\partial_2 F_{23}+\frac{V}{ \text{sin}^2\theta}\partial_3 F_{32}~.
	\end{eqnarray}
From Bianchi identity, we have
	\begin{eqnarray}
	\partial_2 F_{10}&=&\partial_0 F_{12}+\partial_1 F_{20}~,\\
	\partial_3 F_{10}&=&\partial_0 F_{13}+\partial_1 F_{30}~.
	\end{eqnarray}
Using these identities and the Maxwell-Higgs equations to compare the terms of \eqref{com1} and \eqref{com2}, we see that \eqref{maxcom1} holds. Equations \eqref{maxcom2}-\eqref{maxcom4} can also be proved in the same procedure. 
\end{proof}
The compact equations on lemma \ref{maxwellcompact} can be cast into an evolution system and constraint equation as follows:
	\begin{eqnarray}
\label{maxcons1}\partial_t\Phi_1&=&\partial_{r*}\Phi_1+VM\Phi_{0}+if\left( {{D}_{2}}\phi \overline{\phi }-\phi \overline{{{D}_{2}}\phi } \right)+i\frac{f}{\text{sin}\theta}\left( {{D}_{3}}\phi \overline{\phi }-\phi \overline{{{D}_{3}}\phi } \right)~,\\
\label{maxcons2}\partial_t\Phi_0&=&-\partial_{r*}\Phi_0+\bar{M}\Phi_{1}+ir^2\left( {{D}_{0}}\phi \overline{\phi }-\phi \overline{{{D}_{0}}\phi } \right)+ir^2\left( {{D}_{1}}\phi \overline{\phi }-\phi \overline{{{D}_{1}}\phi } \right)~,\\
\label{maxcons3}\partial_t\Phi_{0}&=&\partial_{r*}\Phi_0-M_{1}\Phi_{-1}+ir^2\left( {{D}_{0}}\phi \overline{\phi }-\phi \overline{{{D}_{0}}\phi } \right)+ir^2\left( {{D}_{1}}\phi \overline{\phi }-\phi \overline{{{D}_{1}}\phi } \right)~,\\
\label{maxcons4}\partial_t\Phi_{-1}&=&-\partial_{r*}\Phi_{-1}+V\bar{M}\Phi_{0}+if\left( {{D}_{2}}\phi \overline{\phi }-\phi \overline{{{D}_{2}}\phi } \right)+i\frac{f}{\text{sin}\theta}\left( {{D}_{3}}\phi \overline{\phi }-\phi \overline{{{D}_{3}}\phi } \right)~.\notag\\
\end{eqnarray}
By adding equations \eqref{maxcons2} and \eqref{maxcons3}, we obtain the following result
\begin{eqnarray}
\partial_t\Phi_{0}=\frac{1}{2}\left(\bar{M}\Phi_{1}-M_{1}\Phi_{-1}\right)+ir^2\left( {{D}_{0}}\phi \overline{\phi }-\phi \overline{{{D}_{0}}\phi } \right)+ir^2\left( {{D}_{1}}\phi \overline{\phi }-\phi \overline{{{D}_{1}}\phi } \right)~.
\end{eqnarray}
Next, by subtracting \eqref{maxcons2} and \eqref{maxcons3}, we derive the constraint equation
\begin{eqnarray}
\partial_{r*}\Phi_0=\frac{1}{2}\left(\bar{M}\Phi_{1}+M_{1}\Phi_{-1}\right)~.
\end{eqnarray}
From Maxwell-Higgs compact equation, we can see that $\Phi_0$ satisﬁes the following wave equation:
\begin{eqnarray}\label{mideq}
\partial^2_{t}\Phi_0&=&\partial^2_{r*}\Phi_0-VM_1\bar{M}\Phi_{0}-M_1\left(if\left( {{D}_{2}}\phi \overline{\phi }-\phi \overline{{{D}_{2}}\phi } \right)\right)\notag\\
&& -M_1 \left(i\frac{f}{\text{sin}\theta}\left( {{D}_{3}}\phi \overline{\phi }-\phi \overline{{{D}_{3}}\phi } \right)\right)\notag\\
&&+ L\left(ir^2\left( {{D}_{0}}\phi \overline{\phi }-\phi \overline{{{D}_{0}}\phi } \right)+ir^2\left( {{D}_{1}}\phi \overline{\phi }-\phi \overline{{{D}_{1}}\phi } \right)\right)~.
\end{eqnarray}
To express the Maxwell-Higgs equation in the form of a wave equation, we begin by defining the spherical Laplacian as follows:
\begin{eqnarray}
\Delta_{\mathcal{S}}=\partial^2_\theta+\frac{1}{\text{sin}^2\theta}\partial^2_\varphi+\text{cot}\theta \partial_\theta~.
\end{eqnarray}
Then \eqref{mideq} becomes
\begin{eqnarray}\label{mideq2}
\partial^2_{t}\Phi_0-\partial^2_{r*}\Phi_0+V\Delta_{\mathcal{S}}\Phi_{0}&=&-M_1\left(if\left( {{D}_{2}}\phi \overline{\phi }-\phi \overline{{{D}_{2}}\phi } \right)\right)\notag\\
&&- M_1 \left(i\frac{f}{\text{sin}\theta}\left( {{D}_{3}}\phi \overline{\phi }-\phi \overline{{{D}_{3}}\phi } \right)\right)\notag\\
&&+ L\left(ir^2\left( {{D}_{0}}\phi \overline{\phi }-\phi \overline{{{D}_{0}}\phi } \right)+ir^2\left( {{D}_{1}}\phi \overline{\phi }-\phi \overline{{{D}_{1}}\phi } \right)\right)~.\notag\\
\end{eqnarray}
Let us define symbol $\tilde{\Box}$ such that
\begin{eqnarray}
\tilde{\Box}u\equiv\partial^2_{t}u-\partial^2_{r*}u+V\slashed{\Delta}u~.
\end{eqnarray}
So in this notation, we can write down \eqref{mideq2} in the form 
\begin{eqnarray}\label{boxtilde_u}
\tilde{\Box}u=\rho(\phi, \bar{\phi})~,
\end{eqnarray}
where
\begin{eqnarray}
\rho(\phi, \bar{\phi})&=&-M_1\left(if\left( {{D}_{2}}\phi \overline{\phi }-\phi \overline{{{D}_{2}}\phi } \right)\right)\notag\\
&&- M_1 \left(i\frac{f}{\text{sin}\theta}\left( {{D}_{3}}\phi \overline{\phi }-\phi \overline{{{D}_{3}}\phi } \right)\right)\notag\\
&&+ L\left(ir^2\left( {{D}_{0}}\phi \overline{\phi }-\phi \overline{{{D}_{0}}\phi } \right)+ir^2\left( {{D}_{1}}\phi \overline{\phi }-\phi \overline{{{D}_{1}}\phi } \right)\right)~.
\end{eqnarray}

{\allowdisplaybreaks
\section{The \texorpdfstring{ $L^\infty$}{} Estimate}
\label{sec:estimate}
In this section, we focus on establishing the $L^\infty$ estimates of the scalar field $\phi$ and the vector field $A$. For the solution of \eqref{boxtilde_u}, there are two significant quantities, one conserved and the other controlled. These quantities are related to the time translation vector field $T = \partial_t$ and the Morawetz vector field, 
\begin{eqnarray}\label{K}
K = (t^2 + r^{*2}) \partial_t + 2tr^* \partial_{r^*} u~.
\end{eqnarray}
Next, we define
\begin{eqnarray}
e\equiv\left(\partial_t u\right)^2+\left(\partial_r^* u\right)^2+V|\slashed{\nabla} u|^2+|D_0 \phi|^2+|D_i \phi|^2~,
\end{eqnarray}
\begin{eqnarray}
e_\mathcal{C}\equiv\frac{1}{2}\left(t^2+r*^2\right) e+2tr^*\partial_t u \partial_{r^*}u+e~,
\end{eqnarray}
Subsequently, we define the associated energy and conformal energy as determined by the integrals of their respective densities.
\begin{eqnarray}\label{energy}
E[u,\phi](t)\equiv\frac{1}{2}\int\limits_{{{\Sigma }_{t}}}{ed{{r}^{*}}{{d}^{2}}\omega }~,
\end{eqnarray}
\begin{eqnarray}\label{energyEC}
E_\mathcal{C}[u,\phi](t)\equiv\frac{1}{2}\int\limits_{{{\Sigma }_{t}}}{e_\mathcal{C}d{{r}^{*}}{{d}^{2}}\omega }~,
\end{eqnarray}
where $\Sigma_t=\left\{ t \right\}\times \mathbb{R}\times {\mathcal{S}^{2}}$. 

\begin{lemma}\label{lema2}
The energy solution for \eqref{boxtilde_u} is conserved. In particular, for any $t \in \mathbb{R}$, the energy $E[u, \phi]$ satisfies
    \begin{eqnarray}
        E[u,\phi](t)=E[u,\phi](t_0)=E_0, \forall t\in \mathbb{R}~.
    \end{eqnarray}
\end{lemma}
\begin{proof}
We can write down \eqref{energy} as
\begin{eqnarray}
   E[u,\phi](t)=\frac{1}{2}\int\limits_{{{\Sigma }_{t}}}{\left( {{\left( {{\partial }_{t}}u \right)}^{2}}+{{\left( \partial _{r}^{*}u \right)}^{2}}+V|\slashed{\nabla }u{{|}^{2}}+{{\left| {{D}_{0}}\phi  \right|}^{2}}+{{\left| {{D}_{i}}\phi  \right|}^{2}} \right)d{{r}^{*}}{{d}^{2}}\omega }~.
\end{eqnarray}
Let us consider the first derivative of $E[u](t)$ with respect to time,
\begin{eqnarray}
     \frac{dE[u,\phi](t)}{dt}&=&\int_{{{\Sigma }_{t}}}{\left( ({{\partial }_{t}}u)({{\partial }_{t}}{{\partial }_{t}}u)+({{\partial }_{{{r}^{*}}}}u)({{\partial }_{t}}{{\partial }_{{{r}^{*}}}}u)+V(\slashed{\nabla }u)({{\partial }_{t}}\slashed{\nabla }u) \right.} \notag\\ 
 && +\left. ({{D}_{0}}\phi )({{\partial }_{t}}{{D}_{0}}\phi )+({{D}_{i}}\phi )({{\partial }_{t}}{{D}_{i}}\phi ) \right)d{{r}^{*}}{{d}^{2}}\omega~.
\end{eqnarray}
Using integration by parts and the equations of motion, and given boundary conditions at spacial infinity, we can show that the terms can cancel out. Thus, the energy $E[u,\phi](t)$ is conserved.
\end{proof}
We are now ready to derive the $L^\infty$ estimate for the fields.
\begin{preposition}
    Define 
    \begin{eqnarray}
  {{\left\| f \right\|}_{L_{loc}^{\infty }}}=\underset{\left[ {{t}_{1}},{{t}_{2}} \right]\times \left\{ \frac{1}{2}t\le \left| {{r}_{*}} \right|\le \frac{3}{4}t \right\}\times {{S}^{2}}}{\mathop{\sup }}\,|f| 
\end{eqnarray}
For $F$ and $\phi$ satisfying \eqref{eom1} and \eqref{eom2}, the following inequalities hold:
 \begin{eqnarray}\label{phiinfty}
   {{\left\| \phi  \right\|}_{L_{loc}^{\infty }}}&\le& \tilde{C}\left( 1+t \right) \\ 
 \label{Ainfty}{{\left\| A \right\|}_{L_{loc}^{\infty }}}&\le& \tilde{C}\left( 1+t \right)
\end{eqnarray}
where $\tilde{C}$ is a constant function that depends only on initial data $E_0$.
\end{preposition}
\begin{proof}
To establish the proof of the lemma, we invoke the Sobolev Embedding Theorem on a Riemannian manifold, as outlined in the following result:
\begin{lemma}
    Let $M_n$ be a complete Riemannian manifold with injectivity radius $\delta_0>0$ and sectional curvature $K$, satisfying the bound $K\leq b^2$. There exists a constant $C(q)$ such that for all $g \in \mathcal{D}(M_n)$:
    \begin{eqnarray}
        \sup |g| \leq C(q) \|g\|_{H^q} \quad \text{if} \quad q > n.
    \end{eqnarray}
\end{lemma}
\begin{proof}
 For the proof of this lemma, we direct the reader to see \cite{aubin}   
\end{proof}
Subsequently, by applying the lemma to $\phi$ and $A$, we obtain the following result
\begin{eqnarray}
        {{\left\| \phi  \right\|}_{L_{loc}^{\infty }}} \leq C(4) \|\phi\|_{H^4_{loc}}
    \end{eqnarray}
    \begin{eqnarray}
       {{\left\| A \right\|}_{L_{loc}^{\infty }}} \leq C(4) \|A\|_{H^4_{loc}}
    \end{eqnarray}
where
\begin{eqnarray}\label{phiH4}
    \|\phi\|_{H^4_{loc}}^2 = \|\phi\|_{L^2_{loc}}^2 + \|\nabla \phi\|_{L^2_{loc}}^2 + \|\nabla^2 \phi\|_{L^2_{loc}}^2 + \|\nabla^3 \phi\|_{L^2_{loc}}^2 + \|\nabla^4 \phi\|_{L^2_{loc}}^2~,
\end{eqnarray}
\begin{eqnarray}\label{AH4}
    \|A\|_{H^4_{loc}}^2 = \|A\|_{L^2_{loc}}^2 + \|\nabla A\|_{L^2_{loc}}^2 + \|\nabla^2 A\|_{L^2_{loc}}^2 + \|\nabla^3 A\|_{L^2_{loc}}^2 + \|\nabla^4 A\|_{L^2_{loc}}^2~.
\end{eqnarray}
To get bound on \eqref{phiH4} and \eqref{AH4}, we divide the proof into the following lemmas
\begin{lemma}\label{lemmaL2Aphi}
    Define 
    \begin{eqnarray}
   {{\left\| u \right\|}_{L_{loc}^{p}}}={{\left( \int\limits_{\left[ {{t}_{1}},{{t}_{2}} \right]\times \left\{ \frac{1}{2}t\le \left| {{r}_{*}} \right|\le \frac{3}{4}t \right\}\times {{S}^{2}}}{{{\left| u \right|}^{p}}d\varsigma } \right)}^{1/p}}~.
\end{eqnarray}
For $F$ and $\phi$ satisfying \eqref{eom1} and \eqref{eom2}, the following inequalities hold:
 \begin{eqnarray}\label{estiphiL2}
  {{\left\| \phi  \right\|}_{L_{loc}^{2}}}\le \Tilde{C}\left( 1+t \right)~,
\end{eqnarray}
 \begin{eqnarray}\label{estiAL2}
  {{\left\| A  \right\|}_{L_{loc}^{2}}}\le \Tilde{C}\left( 1+t \right)~.
\end{eqnarray}
\end{lemma}
\begin{proof}
    Let us consider the derivative of $\left\| \phi  \right\|_{L_{loc}^{2}}^{2}$ with respect to $t$
\begin{eqnarray}
    \frac{\partial \left\| \phi  \right\|_{L_{loc}^{2}}^{2}}{\partial t}=\int\limits_{\left[ {{t}_{1}},{{t}_{2}} \right]\times \left\{ \frac{1}{2}t\le \left| {{r}_{*}} \right|\le \frac{3}{4}t \right\}\times {{S}^{2}}}{2\phi \frac{\partial \phi }{\partial t}d\varsigma }~,
\end{eqnarray}
Using H\"older inequality we have
\begin{eqnarray}
  2{{\left\| \phi  \right\|}_{L_{loc}^{2}}}\frac{\partial {{\left\| \phi  \right\|}_{L_{loc}^{2}}}}{\partial t}&\le& 2{{\left( \int\limits_{{{\Omega }_{loc}}}{{{\left| \phi  \right|}^{2}}d\varsigma } \right)}^{1/2}}{{\left( \int\limits_{{{\Omega }_{loc}}}{{{\left| \frac{\partial \phi }{\partial t} \right|}^{2}}d\varsigma } \right)}^{1/2}} \notag\\ 
  \frac{\partial {{\left\| \phi  \right\|}_{L_{loc}^{2}}}}{\partial t}&\le& {{\left( \int\limits_{{{\Omega }_{loc}}}{{{\left| \frac{\partial \phi }{\partial t} \right|}^{2}}d\varsigma } \right)}^{1/2}} \notag\\
    & \le& {{\left( \int\limits_{{{\Omega }_{loc}}}{{{\left| \operatorname{Re}\left( {{D}_{0}}\phi  \right) \right|}^{2}}d\varsigma } \right)}^{1/2}} \\ 
 & \le& E_{0}^{1/2} ~,
\end{eqnarray}
where for simplicity ${{\Omega }_{loc}}=\left[ {{t}_{1}},{{t}_{2}} \right]\times \left\{ \frac{1}{2}t\le \left| {{r}_{*}} \right|\le \frac{3}{4}t \right\}\times {{S}^{2}}$.
Thus, integrating over $t$, we obtain
\begin{eqnarray}
    {{\left\| \phi  \right\|}_{L_{loc}^{2}}}\le CE_{0}^{1/2}\left( 1+t \right)~.
\end{eqnarray}
This completes the proof of \eqref{estiphiL2}. The same procedure also applies to prove \eqref{estiAL2}, 
\begin{eqnarray}
  \frac{\partial \left\| A \right\|_{L_{loc}^{2}}^{2}}{\partial t}&\le& 2{{\left( \int\limits_{{{\Omega }_{loc}}}{{{\left| A \right|}^{2}}d\varsigma } \right)}^{1/2}}{{\left( \int\limits_{{{\Omega }_{loc}}}{{{\left| \frac{\partial A}{\partial t} \right|}^{2}}d\varsigma } \right)}^{1/2}} \notag\\ 
 \frac{\partial {{\left\| A \right\|}_{L_{loc}^{2}}}}{\partial t}&\le& {{\left( \int\limits_{{{\Omega }_{loc}}}{{{\left| \frac{\partial A}{\partial t} \right|}^{2}}d\varsigma } \right)}^{1/2}} \notag\\ 
 & \le& {{\left( \int\limits_{{{\Omega }_{loc}}}{{{\left| {{F}_{0i}} \right|}^{2}}d\varsigma } \right)}^{1/2}} \notag\\ 
 & \le& E_{0}^{1/2} \notag\\ 
 {{\left\| A \right\|}_{L_{loc}^{2}}}&\le& CE_{0}^{1/2}\left( 1+t \right)  
\end{eqnarray}
\end{proof}

\begin{lemma}\label{lemmaderif1Aphi}
Let ${{\left\| u \right\|}_{L_{loc}^{p}}}$ as in the definition in lemma \ref{lemmaL2Aphi} and $\Tilde{C}$ is a function that only depend on the initial data $E_0$. For $F$ and $\phi$ satisfying \eqref{eom1} and \eqref{eom2}, the following inequalities hold:
 \begin{eqnarray}\label{estiderif1phiL2}
  {{\left\| \nabla \phi  \right\|}_{L_{loc}^{2}}}\le \tilde{C} ~,
\end{eqnarray}
 \begin{eqnarray}\label{estiderif1AL2}
  {{\left\| \nabla A \right\|}_{L_{loc}^{2}}}\le \tilde{C} ~,
\end{eqnarray}
 \begin{eqnarray}\label{estiderif2AL2}
  {{\left\| \nabla\nabla A \right\|}_{L_{loc}^{2}}}\le \tilde{C} {{\left\| \phi  \right\|}_{L_{loc}^{\infty}}}~,
\end{eqnarray}
\begin{eqnarray}\label{estiderif2phiL2}
  {{\left\| \nabla\nabla \phi  \right\|}_{L_{loc}^{2}}}\le \tilde{C}\left({{\left\|  A  \right\|}_{L_{loc}^{\infty}}}+{{\left\| \phi  \right\|}_{L_{loc}^{\infty}}}\right) ~.
\end{eqnarray}
\end{lemma}
\begin{proof}
The proof of \eqref{estiderif1phiL2} and \eqref{estiderif1AL2} are straightforward because ${{\left\| \nabla \phi  \right\|}_{L_{loc}^{2}}}$ and ${{\left\| \nabla A \right\|}_{L_{loc}^{2}}}$ are already bounded within the energy term $E_0$. To prove \eqref{estiderif2AL2}, we apply \eqref{eom1} and using the H\"older inequality to get
\begin{eqnarray}
   {{\left\| \nabla \nabla A \right\|}_{L_{loc}^{2}}}&\le& {{\left\| \phi D\phi  \right\|}_{L_{loc}^{2}}} \notag\\ 
 & \le& {{\left\| \phi  \right\|}_{L_{loc}^{\infty}}}{{\left\| D\phi  \right\|}_{L_{loc}^{2}}} \notag\\
  & \le& E_0^{1/2}{{\left\| \phi  \right\|}_{L_{loc}^{\infty}}} \notag\\
  & \le& \tilde{C}{{\left\| \phi  \right\|}_{L_{loc}^{\infty}}}
\end{eqnarray}
Then, using the same procedure for the scalar field, we arrive at
\begin{eqnarray}
    {{\left\| \nabla \nabla \phi \right\|}_{L_{loc}^{2}}}&\le& {{\left\| \nabla\phi A  \right\|}_{L_{loc}^{2}}}+{{\left\| \phi \nabla A  \right\|}_{L_{loc}^{2}}}+{{\left\|A D\phi    \right\|}_{L_{loc}^{2}}}\notag\\ 
 & \le& {{\left\| \nabla\phi  \right\|}_{L_{loc}^{2}}}{{\left\|  A  \right\|}_{L_{loc}^{\infty}}}+{{\left\| \phi  \right\|}_{L_{loc}^{\infty}}}{{\left\| \nabla A  \right\|}_{L_{loc}^{2}}}+{{\left\|A    \right\|}_{L_{loc}^{\infty}}}{{\left\| D\phi    \right\|}_{L_{loc}^{2}}}\notag\\
  & \le& \tilde{C}\left({{\left\|  A  \right\|}_{L_{loc}^{\infty}}}+{{\left\| \phi  \right\|}_{L_{loc}^{\infty}}}\right)
\end{eqnarray}
\end{proof}


\begin{lemma}\label{lemmaderif3Aphi}
Let ${{\left\| u \right\|}_{L_{loc}^{p}}}$ be as specified in Lemma \ref{lemmaL2Aphi} and $\Tilde{C}$ be a function solely determined by the initial data $E_0$. For $F$ and $\phi$ satisfying \eqref{eom1} and \eqref{eom2}, the following inequalities are satisfied:
 \begin{eqnarray}\label{estiderif3AL2}
   {{\left\| \nabla \nabla \nabla A \right\|}_{L_{loc}^{2}}}\le \tilde{C}{\left( 1+t \right)}Y~,
\end{eqnarray}
\begin{eqnarray}\label{estiderif3phiL2}
  {{\left\| \nabla\nabla\nabla \phi  \right\|}_{L_{loc}^{2}}}\le \tilde{C}\left( 1+t \right)Z~,
\end{eqnarray}
   \begin{eqnarray}\label{estiderif4AL2}
   {{\left\| \nabla\nabla \nabla \nabla A \right\|}_{L_{loc}^{2}}}\le \tilde{C}{\left( 1+t \right)}X~,
\end{eqnarray}
\begin{eqnarray}\label{estiderif4phiL2}
  {{\left\| \nabla\nabla\nabla\nabla \phi  \right\|}_{L_{loc}^{2}}}\le \tilde{C}\left( 1+t \right)W~.
\end{eqnarray}
where
\begin{eqnarray}\label{Y}
   Y&=&\left( {{\left\| A \right\|}_{L_{loc}^{\infty }}}+{{\left\| \phi  \right\|}_{L_{loc}^{\infty }}} \right)\left( {{\left\| A \right\|}_{L_{loc}^{\infty }}}+{{\left\| \phi  \right\|}_{L_{loc}^{\infty }}}+{{\left( {{\left\| A \right\|}_{L_{loc}^{\infty }}}+{{\left\| \phi  \right\|}_{L_{loc}^{\infty }}} \right)}^{3/4}}+\left\| \phi  \right\|_{L_{loc}^{\infty }}^{3/2} \right) \notag\\ 
 &&+\left\| \phi  \right\|_{L_{loc}^{\infty }}^{2}+{{\left\| \phi  \right\|}_{L_{loc}^{\infty }}}{{\left\| A \right\|}_{L_{loc}^{\infty }}}+{{\left\| \phi  \right\|}_{L_{loc}^{\infty }}}\left( {{\left\| A \right\|}_{L_{loc}^{\infty }}}+{{\left\| \phi  \right\|}_{L_{loc}^{\infty }}} \right) ~,  
\end{eqnarray}
\begin{eqnarray}
   Z&=&{{\left( {{\left\| A \right\|}_{L_{loc}^{\infty }}}+{{\left\| \phi  \right\|}_{L_{loc}^{\infty }}} \right)}^{3/4}}\left\| \phi  \right\|_{L_{loc}^{\infty }}^{3/4}+{{\left\| A \right\|}_{L_{loc}^{\infty }}}+{{\left\| \phi  \right\|}_{L_{loc}^{\infty }}} \notag\\ 
 && +\left\| \phi  \right\|_{L_{loc}^{\infty }}^{3/4}{{\left\| A \right\|}_{L_{loc}^{\infty }}}+\left\| \phi  \right\|_{L_{loc}^{\infty }}^{3/4}{{\left\| \phi  \right\|}_{L_{loc}^{\infty }}}+\left\| \phi  \right\|_{L_{loc}^{\infty }}^{3/4}{{\left( {{\left\| A \right\|}_{L_{loc}^{\infty }}}+{{\left\| \phi  \right\|}_{L_{loc}^{\infty }}} \right)}^{3/4}} \notag\\ 
 && +\left\| \phi  \right\|_{L_{loc}^{\infty }}^{9/4}+\left\| \phi  \right\|_{L_{loc}^{\infty }}^{2}+\left\| A \right\|_{L_{loc}^{\infty }}^{2}+{{\left\| \phi  \right\|}_{L_{loc}^{\infty }}}{{\left\| A \right\|}_{L_{loc}^{\infty }}} \notag\\ 
 && +{{\left\| A \right\|}_{L_{loc}^{\infty }}}{{\left( {{\left\| A \right\|}_{L_{loc}^{\infty }}}+{{\left\| \phi  \right\|}_{L_{loc}^{\infty }}} \right)}^{3/4}}+{{\left\| A \right\|}_{L_{loc}^{\infty }}}\left\| \phi  \right\|_{L_{loc}^{\infty }}^{3/2}~,   
\end{eqnarray}
\begin{eqnarray}\label{X}
   X&=&Z\left\{ \left( {{\left\| A \right\|}_{L_{loc}^{\infty }}}+{{\left\| \phi  \right\|}_{L_{loc}^{\infty }}} \right)+{{\left( {{\left\| A \right\|}_{L_{loc}^{\infty }}}+{{\left\| \phi  \right\|}_{L_{loc}^{\infty }}} \right)}^{3/2}} \right. \notag\\ 
 && +{{\left( {{\left\| A \right\|}_{L_{loc}^{\infty }}}+{{\left\| \phi  \right\|}_{L_{loc}^{\infty }}} \right)}^{1/4}}+{{\left\| \phi  \right\|}_{L_{loc}^{\infty }}}+{{\left\| \phi  \right\|}_{L_{loc}^{\infty }}}{{\left\| A \right\|}_{L_{loc}^{\infty }}} \notag\\ 
 && \left. +{{\left( {{\left\| A \right\|}_{L_{loc}^{\infty }}}+{{\left\| \phi  \right\|}_{L_{loc}^{\infty }}} \right)}^{3/4}}\left\| \phi  \right\|_{L_{loc}^{2}}^{2} \right\} \notag\\ 
 && +{{\left\| \phi  \right\|}_{L_{loc}^{\infty }}}{{\left\| A \right\|}_{L_{loc}^{\infty }}}\left( {{\left\| A \right\|}_{L_{loc}^{\infty }}}+{{\left\| \phi  \right\|}_{L_{loc}^{\infty }}} \right)+{{\left( {{\left\| A \right\|}_{L_{loc}^{\infty }}}+{{\left\| \phi  \right\|}_{L_{loc}^{\infty }}} \right)}^{3/4}} \notag\\ 
 && +{{\left( {{\left\| A \right\|}_{L_{loc}^{\infty }}}+{{\left\| \phi  \right\|}_{L_{loc}^{\infty }}} \right)}^{3/2}}{{\left\| A \right\|}_{L_{loc}^{\infty }}}~,  
\end{eqnarray}
\begin{eqnarray}
   W&=&{{Z}^{3/4}}{{\left( {{\left\| A \right\|}_{L_{loc}^{\infty }}}+{{\left\| \phi  \right\|}_{L_{loc}^{\infty }}} \right)}^{1/4}}\left\| \phi  \right\|_{L_{loc}^{\infty }}^{3/4}+{{Y}^{3/4}}\left\| \phi  \right\|_{L_{loc}^{\infty }}^{1/4}{{\left( {{\left\| A \right\|}_{L_{loc}^{\infty }}}+{{\left\| \phi  \right\|}_{L_{loc}^{\infty }}} \right)}^{3/4}}\notag\\ 
 && +Z{{\left\| A \right\|}_{L_{loc}^{\infty }}}+\left( 1+t \right)Y{{\left\| \phi  \right\|}_{L_{loc}^{\infty }}}+Z\left\| \phi  \right\|_{L_{loc}^{\infty }}^{3/4}{{\left( {{\left\| A \right\|}_{L_{loc}^{\infty }}}+{{\left\| \phi  \right\|}_{L_{loc}^{\infty }}} \right)}^{1/4}} \notag\\ 
 &&  +{{\left\| \phi  \right\|}_{L_{loc}^{\infty }}}\left\| \phi  \right\|_{L_{loc}^{\infty }}^{3/2}+\left\| \phi  \right\|_{L_{loc}^{\infty }}^{3/4}{{\left\| A \right\|}_{L_{loc}^{\infty }}}{{\left( {{\left\| A \right\|}_{L_{loc}^{\infty }}}+{{\left\| \phi  \right\|}_{L_{loc}^{\infty }}} \right)}^{3/4}} \notag\\ 
 & & +{{\left\| A \right\|}_{L_{loc}^{\infty }}}Z+{{\left\| A \right\|}_{L_{loc}^{\infty }}}\left\| \phi  \right\|_{L_{loc}^{\infty }}^{2}+\left\| A \right\|_{L_{loc}^{\infty }}^{2}\left( {{\left\| A \right\|}_{L_{loc}^{\infty }}}+{{\left\| \phi  \right\|}_{L_{loc}^{\infty }}} \right) \notag\\ 
 & & +\left\| \phi  \right\|_{L_{loc}^{\infty }}^{3/4}{{\left\| A \right\|}_{L_{loc}^{\infty }}}{{\left( {{\left\| A \right\|}_{L_{loc}^{\infty }}}+{{\left\| \phi  \right\|}_{L_{loc}^{\infty }}} \right)}^{1/4}} ~. 
\end{eqnarray}
\end{lemma}
\begin{proof}
To prove \eqref{estiderif3AL2}, we derive \eqref{eom1} and using \eqref{estiderif1phiL2}-\eqref{estiderif2phiL2} also the Minkowski inequality and Gagliardo-Nirenberg-Sobolev inequality, we have
\begin{eqnarray}
  {{\left\| \nabla \nabla \nabla A \right\|}_{L_{loc}^{2}}}&\le& {{\left\| \nabla \phi D\phi  \right\|}_{L_{loc}^{2}}}+{{\left\| \phi \nabla \nabla \phi  \right\|}_{L_{loc}^{2}}}+{{\left\| {{\phi }^{2}}\nabla A \right\|}_{L_{loc}^{2}}}+{{\left\| \phi A\nabla \phi  \right\|}_{L_{loc}^{2}}} \notag\\ 
 & \le& {{\left\| \nabla \phi  \right\|}_{L_{loc}^{4}}}{{\left\| D\phi  \right\|}_{L_{loc}^{4}}}+{{\left\| \phi  \right\|}_{L_{loc}^{\infty }}}{{\left\| \nabla \nabla \phi  \right\|}_{L_{loc}^{2}}} \notag\\ 
 && +\left\| \phi  \right\|_{L_{loc}^{\infty }}^{2}{{\left\| \nabla A \right\|}_{L_{loc}^{2}}}+{{\left\| \phi  \right\|}_{L_{loc}^{\infty }}}{{\left\| A \right\|}_{L_{loc}^{\infty }}}{{\left\| \nabla \phi  \right\|}_{L_{loc}^{2}}} \notag\\ 
 & \le& \tilde{C}{{\left( 1+t \right)}^{1/4}}\left\{ \left( {{\left\| A \right\|}_{L_{loc}^{\infty }}}+{{\left\| \phi  \right\|}_{L_{loc}^{\infty }}}+{{\left( {{\left\| A \right\|}_{L_{loc}^{\infty }}}+{{\left\| \phi  \right\|}_{L_{loc}^{\infty }}} \right)}^{3/4}}+\left\| \phi  \right\|_{L_{loc}^{\infty }}^{3/2} \right) \right. \notag\\ 
 && \cdot \left( {{\left\| A \right\|}_{L_{loc}^{\infty }}}+{{\left\| \phi  \right\|}_{L_{loc}^{\infty }}} \right)+\left\| \phi  \right\|_{L_{loc}^{\infty }}^{2}+{{\left\| \phi  \right\|}_{L_{loc}^{\infty }}}{{\left\| A \right\|}_{L_{loc}^{\infty }}} \notag\\ 
 & &+\left. {{\left\| \phi  \right\|}_{L_{loc}^{\infty }}}\left( {{\left\| A \right\|}_{L_{loc}^{\infty }}}+{{\left\| \phi  \right\|}_{L_{loc}^{\infty }}} \right) \right\}  
\end{eqnarray}
By defining $Y$ as in \eqref{Y}, we arrive at equation \eqref{estiderif3AL2}.
To prove \eqref{estiderif3phiL2}, we employ the same procedure. Hence, we have
\begin{eqnarray}
  {{\left\| \nabla \nabla \nabla \phi  \right\|}_{L_{loc}^{2}}}&\le& {{\left\| \nabla \phi \nabla A \right\|}_{L_{loc}^{2}}}+{{\left\| \nabla \nabla \phi A \right\|}_{L_{loc}^{2}}}+{{\left\| \phi \nabla \nabla A \right\|}_{L_{loc}^{2}}}+{{\left\| \nabla AD\phi  \right\|}_{L_{loc}^{2}}}+{{\left\| A\nabla D\phi  \right\|}_{L_{loc}^{2}}} \notag\\ 
 & \le& {{\left\| \nabla \phi  \right\|}_{L_{loc}^{4}}}{{\left\| \nabla A \right\|}_{L_{loc}^{4}}}+{{\left\| \nabla \nabla \phi  \right\|}_{L_{loc}^{2}}}{{\left\| A \right\|}_{L_{loc}^{\infty }}}+{{\left\| \nabla \nabla A \right\|}_{L_{loc}^{2}}}{{\left\| \phi  \right\|}_{L_{loc}^{\infty }}}  \notag\\ 
 && +{{\left\| \nabla A \right\|}_{L_{loc}^{4}}}{{\left\| D\phi  \right\|}_{L_{loc}^{4}}}+{{\left\| A \right\|}_{L_{loc}^{\infty }}}{{\left\| \nabla D\phi  \right\|}_{L_{loc}^{2}}}  \notag\\ 
 & \le& \left\| \nabla \nabla \phi  \right\|_{L_{loc}^{2}}^{3/4}\left\| \nabla \phi  \right\|_{L_{loc}^{2}}^{1/4}\left\| \nabla \nabla A \right\|_{L_{loc}^{2}}^{3/4}\left\| \nabla A \right\|_{L_{loc}^{2}}^{1/4}+{{\left\| \nabla \nabla \phi  \right\|}_{L_{loc}^{2}}}{{\left\| A \right\|}_{L_{loc}^{\infty }}}  \notag\\ 
 & &+\left\| \nabla \nabla A \right\|_{L_{loc}^{2}}^{3/4}\left\| \nabla A \right\|_{L_{loc}^{2}}^{1/4}\left( {{\left\| \nabla \nabla \phi  \right\|}_{L_{loc}^{2}}}+{{\left\| \nabla \phi A \right\|}_{L_{loc}^{2}}}+{{\left\| \phi \nabla A \right\|}_{L_{loc}^{2}}} \right) \notag \\ 
 & &+{{\left\| \nabla \nabla A \right\|}_{L_{loc}^{2}}}{{\left\| \phi  \right\|}_{L_{loc}^{\infty }}}+{{\left\| A \right\|}_{L_{loc}^{\infty }}}\left( {{\left\| \nabla \nabla \phi  \right\|}_{L_{loc}^{2}}}+{{\left\| \nabla A\phi  \right\|}_{L_{loc}^{2}}}+{{\left\| A\nabla \phi  \right\|}_{L_{loc}^{2}}} \right)  \notag\\ 
 & \le& \tilde{C}\left( 1+t \right)\left\{ {{\left( {{\left\| A \right\|}_{L_{loc}^{\infty }}}+{{\left\| \phi  \right\|}_{L_{loc}^{\infty }}} \right)}^{3/4}}\left\| \phi  \right\|_{L_{loc}^{\infty }}^{3/4}+{{\left\| A \right\|}_{L_{loc}^{\infty }}}+{{\left\| \phi  \right\|}_{L_{loc}^{\infty }}} \right.  \notag\\ 
 && +\left\| \phi  \right\|_{L_{loc}^{\infty }}^{3/4}{{\left\| A \right\|}_{L_{loc}^{\infty }}}+\left\| \phi  \right\|_{L_{loc}^{\infty }}^{3/4}{{\left\| \phi  \right\|}_{L_{loc}^{\infty }}}+\left\| \phi  \right\|_{L_{loc}^{\infty }}^{3/4}{{\left( {{\left\| A \right\|}_{L_{loc}^{\infty }}}+{{\left\| \phi  \right\|}_{L_{loc}^{\infty }}} \right)}^{3/4}}  \notag\\ 
 && +\left\| \phi  \right\|_{L_{loc}^{\infty }}^{9/4}+\left\| \phi  \right\|_{L_{loc}^{\infty }}^{2}+\left\| A \right\|_{L_{loc}^{\infty }}^{2}+{{\left\| \phi  \right\|}_{L_{loc}^{\infty }}}{{\left\| A \right\|}_{L_{loc}^{\infty }}} \notag\\ 
 &&+ \left. {{\left\| A \right\|}_{L_{loc}^{\infty }}}{{\left( {{\left\| A \right\|}_{L_{loc}^{\infty }}}+{{\left\| \phi  \right\|}_{L_{loc}^{\infty }}} \right)}^{3/4}}+{{\left\| A \right\|}_{L_{loc}^{\infty }}}\left\| \phi  \right\|_{L_{loc}^{\infty }}^{3/2} \right\}  
\end{eqnarray}
Next, to prove \eqref{estiderif4AL2}, we apply a similar approach as the one used in the proof of the preceding lemma.
    \begin{eqnarray}
  {{\left\| \nabla \nabla \nabla \nabla A \right\|}_{L_{loc}^{2}}}&\le& {{\left\| \nabla \nabla \phi D\phi  \right\|}_{L_{loc}^{2}}}+{{\left\| \nabla \phi \nabla D\phi  \right\|}_{L_{loc}^{2}}}+{{\left\| \nabla \phi \nabla \nabla \phi  \right\|}_{L_{loc}^{2}}} \notag\\ 
 && +{{\left\| \phi \nabla \nabla \nabla \phi  \right\|}_{L_{loc}^{2}}}+{{\left\| \nabla {{\phi }^{2}}\nabla A \right\|}_{L_{loc}^{2}}}+{{\left\| {{\phi }^{2}}\nabla \nabla A \right\|}_{L_{loc}^{2}}} \notag\\ 
 && +{{\left\| \nabla \phi A\nabla \phi  \right\|}_{L_{loc}^{2}}}+{{\left\| \phi \nabla A\nabla \phi  \right\|}_{L_{loc}^{2}}}+{{\left\| \phi A\nabla \nabla \phi  \right\|}_{L_{loc}^{2}}} \\ 
 & \le& \left\| \nabla \nabla \nabla \phi  \right\|_{L_{loc}^{2}}^{3/4}\left\| \nabla \nabla \phi  \right\|_{L_{loc}^{2}}^{1/4}\left\| \nabla D\phi  \right\|_{L_{loc}^{2}}^{3/4}\left\| D\phi  \right\|_{L_{loc}^{2}}^{1/4} \notag\\ 
 & &+\left\| \nabla \nabla \phi  \right\|_{L_{loc}^{2}}^{3/4}\left\| \nabla \phi  \right\|_{L_{loc}^{2}}^{1/4}\left\| \nabla \nabla D\phi  \right\|_{L_{loc}^{2}}^{3/4}\left\| \nabla D\phi  \right\|_{L_{loc}^{2}}^{1/4} \notag\\ 
 & &+\left\| \nabla \nabla \nabla \phi  \right\|_{L_{loc}^{2}}^{3/4}\left\| \nabla \nabla \phi  \right\|_{L_{loc}^{2}}^{1/4}\left\| \nabla \phi  \right\|_{L_{loc}^{2}}^{3/4}\left\| \nabla \phi  \right\|_{L_{loc}^{2}}^{1/4} \notag\\ 
 & &+{{\left\| \nabla \nabla \nabla \phi  \right\|}_{L_{loc}^{2}}}{{\left\| \phi  \right\|}_{L_{loc}^{\infty }}}+{{\left\| \phi  \right\|}_{L_{loc}^{\infty }}}{{\left\| A \right\|}_{L_{loc}^{\infty }}}{{\left\| \nabla \phi  \right\|}_{L_{loc}^{2}}} \notag\\ 
 & &+\left\| \nabla \phi  \right\|_{L_{loc}^{2}}^{3/4}\left\| \nabla \phi  \right\|_{L_{loc}^{2}}^{1/4}\left\| \nabla \phi  \right\|_{L_{loc}^{2}}^{3/4}\left\| \phi  \right\|_{L_{loc}^{2}}^{1/4}{{\left\| \nabla A \right\|}_{L_{loc}^{2}}} \notag\\ 
 & &+\left\| \nabla \nabla \phi  \right\|_{L_{loc}^{2}}^{3/4}\left\| \nabla \phi  \right\|_{L_{loc}^{2}}^{1/4}\left\| \nabla \nabla \nabla A \right\|_{L_{loc}^{2}}^{3/4}\left\| \nabla \nabla A \right\|_{L_{loc}^{2}}^{1/4}{{\left\| \phi  \right\|}_{L_{loc}^{\infty }}} \notag\\ 
 & &+\left\| \nabla \nabla \phi  \right\|_{L_{loc}^{2}}^{3/2}\left\| \nabla \phi  \right\|_{L_{loc}^{2}}^{1/2}{{\left\| A \right\|}_{L_{loc}^{\infty }}} \notag\\ 
 & &+{{\left\| \phi  \right\|}_{L_{loc}^{\infty }}}{{\left\| A \right\|}_{L_{loc}^{\infty }}}{{\left\| \nabla \nabla \phi  \right\|}_{L_{loc}^{2}}}  
    \end{eqnarray}
By using \eqref{estiderif1phiL2}-\eqref{estiderif3phiL2}, we can have \eqref{estiderif4AL2}. The same method also apply to prove \eqref{estiderif4phiL2},
\begin{eqnarray}
   {{\left\| \nabla \nabla \nabla \nabla \phi  \right\|}_{L_{loc}^{2}}}&\le& {{\left\| \nabla \nabla \phi \nabla A \right\|}_{L_{loc}^{2}}}+{{\left\| \nabla \phi \nabla \nabla A \right\|}_{L_{loc}^{2}}}+{{\left\| \nabla \nabla \nabla \phi A \right\|}_{L_{loc}^{2}}} \notag\\ 
 && +{{\left\| \phi \nabla \nabla \nabla A \right\|}_{L_{loc}^{2}}}+{{\left\| \nabla \nabla AD\phi  \right\|}_{L_{loc}^{2}}}+{{\left\| \nabla A\nabla D\phi  \right\|}_{L_{loc}^{2}}} \notag\\ 
 && +{{\left\| A\nabla \nabla D\phi  \right\|}_{L_{loc}^{2}}} \notag\\ 
 & \le& \left\| \nabla \nabla \nabla \phi  \right\|_{L_{loc}^{2}}^{3/4}\left\| \nabla \nabla \phi  \right\|_{L_{loc}^{2}}^{1/4}\left\| \nabla \nabla A \right\|_{L_{loc}^{2}}^{3/4}\left\| \nabla A \right\|_{L_{loc}^{2}}^{1/4} \notag\\ 
 & &+\left\| \nabla \nabla \nabla A \right\|_{L_{loc}^{2}}^{3/4}\left\| \nabla \nabla A \right\|_{L_{loc}^{2}}^{1/4}\left\| \nabla \nabla \phi  \right\|_{L_{loc}^{2}}^{3/4}\left\| \nabla \phi  \right\|_{L_{loc}^{2}}^{1/4} \notag\\ 
 & &+{{\left\| \nabla \nabla \nabla \phi  \right\|}_{L_{loc}^{2}}}{{\left\| A \right\|}_{L_{loc}^{\infty }}}+{{\left\| \nabla \nabla \nabla A \right\|}_{L_{loc}^{2}}}{{\left\| \phi  \right\|}_{L_{loc}^{\infty }}} \notag\\ 
 & &+\left\| \nabla \nabla \nabla A \right\|_{L_{loc}^{2}}^{3/4}\left\| \nabla \nabla A \right\|_{L_{loc}^{2}}^{1/4}\left\| D\phi  \right\|_{L_{loc}^{2}}^{1/4} \notag\\ 
 & &\cdot {{\left( {{\left\| \nabla \nabla \phi  \right\|}_{L_{loc}^{2}}}+{{\left\| \nabla A\phi  \right\|}_{L_{loc}^{2}}}+{{\left\| A\nabla \phi  \right\|}_{L_{loc}^{2}}} \right)}^{3/4}} \notag\\ 
 & &+\left\| \nabla \nabla A \right\|_{L_{loc}^{2}}^{3/4}\left\| \nabla A \right\|_{L_{loc}^{2}}^{1/4}\left( \left\| \nabla \nabla \nabla \phi  \right\|_{L_{loc}^{2}}^{3/4}\left\| \nabla \nabla \phi  \right\|_{L_{loc}^{2}}^{1/4} \right. \notag\\ 
 & &\left. +{{\left\| \phi  \right\|}_{L_{loc}^{\infty }}}\left\| \nabla \nabla A \right\|_{L_{loc}^{2}}^{3/4}\left\| \nabla A \right\|_{L_{loc}^{2}}^{1/4}+{{\left\| A \right\|}_{L_{loc}^{\infty }}}\left\| \nabla \nabla \phi  \right\|_{L_{loc}^{2}}^{3/4}\left\| \nabla \phi  \right\|_{L_{loc}^{2}}^{1/4} \right) \notag\\ 
 & &+{{\left\| A \right\|}_{L_{loc}^{\infty }}}\left( {{\left\| \nabla \nabla \nabla \phi  \right\|}_{L_{loc}^{2}}}+{{\left\| \nabla \nabla A \right\|}_{L_{loc}^{2}}}{{\left\| \phi  \right\|}_{L_{loc}^{\infty }}}+{{\left\| A \right\|}_{L_{loc}^{\infty }}}{{\left\| \nabla \nabla \phi  \right\|}_{L_{loc}^{2}}} \right. \notag\\ 
 & &+\left. {{\left\| \nabla A\nabla \phi  \right\|}_{L_{loc}^{2}}}\left\| \nabla \nabla A \right\|_{L_{loc}^{2}}^{3/4}\left\| \nabla A \right\|_{L_{loc}^{2}}^{1/4}\left\| \nabla \nabla \phi  \right\|_{L_{loc}^{2}}^{3/4}\left\| \nabla \phi  \right\|_{L_{loc}^{2}}^{1/4} \right)  
\end{eqnarray}
By using \eqref{estiderif1phiL2}-\eqref{estiderif3phiL2}, after some computation, we can prove \eqref{estiderif4AL2}.
\end{proof}

Therefore, by utilizing the estimates for each term in equations \eqref{AH4} and \eqref{phiH4} derived earlier, we arrive at the following result
\begin{eqnarray}\label{phiLinfiny}
    {{\left\| \phi  \right\|}_{L_{loc}^{\infty }}}\le \tilde{C}\left( 1+t \right)\left\{ 1+\left( {{\left\| A \right\|}_{L_{loc}^{\infty }}}+{{\left\| \phi  \right\|}_{L_{loc}^{\infty }}} \right)+Z+W \right\}
\end{eqnarray}
\begin{eqnarray}\label{ALinfiny}
    {{\left\| A \right\|}_{L_{loc}^{\infty }}}\le \tilde{C}\left( 1+t \right)\left\{ 1+{{\left\| \phi  \right\|}_{L_{loc}^{\infty }}}+Y+X \right\}
\end{eqnarray}
Alternatively, this allows us to get \eqref{phiinfty} and \eqref{Ainfty}. This completes the proof of preposition \ref{statprep1}.
\end{proof}

}
\section{The Uniform Bound on The Energy}
\label{sec:main}
Our main result is summarised in the following theorem:
\begin{theorem}\label{statprep1}
	Let $u$ be a smooth solution of \eqref{boxtilde_u} and ${E}_{\mathcal{C}}(t)$ denote the conformal energy as defined in \eqref{energyEC}. There exist finite constant $\bar{C}$ such that for a finite time interval $t\ge 0$, the following inequality holds
	\begin{eqnarray}\label{prep1}
	{{E}_{\mathcal{C}}}(t) \le \bar{C}(1+t)^2\left( {{E}_{\mathcal{C}}}(0)+E\left[ \slashed{\Delta^2 }u \right]\left( 0 \right)+E[u](0) \right)~.
	\end{eqnarray}
where $E(0)$ denotes the initial energy of the system.
\end{theorem}
\begin{proof}
In order to get a bound on the conformal energy $E_{\mathcal{C}}$, we use the Morawetz multiplier $K(u)$ and then we integrate by parts as demonstrated in the proof of the next lemma. 

\begin{lemma}
Let $u$ be a solution to \eqref{boxtilde_u}. There exists a non-negative, compactly supported, smooth function $\chi$ of $r_*$ such that for all $t_2, t_1 > 0$, the following inequality holds:
\begin{eqnarray}\label{lemma3}
E_\mathcal{C}[u,\phi](t_2)-E_\mathcal{C}[u,\phi](t_1)\le \int\limits_{\left[ {{t}_{1}},{{t}_{2}} \right]\times \mathbb{R}\times {{S}^{2}}}t{{\chi }}{{\left| \slashed{\nabla }u \right|}^{2}}dtd{{r}_{*}}{{d}^{2}}\omega +\int\limits_{{\left[ {{t}_{1}},{{t}_{2}} \right]\times \mathbb{R}\times {{S}^{2}}}}{ \mathcal{Q} \left( \phi ,\bar{\phi } \right)dx}~,\notag\\
\end{eqnarray} 
\end{lemma}
\begin{proof}
    Let us proof the lemma using the multiplier method. We define
    \begin{eqnarray}\label{Kbox}
        K(u)\tilde{\Box}u= \mathcal{Q}\left(\phi, \Bar{\phi}\right)~.
    \end{eqnarray}
Then, if we integrate \eqref{Kbox} over the domain $\Omega={\left[ {{t}_{1}},{{t}_{2}} \right]\times \mathbb{R}\times {{S}^{2}}}$, we have
\begin{eqnarray}\label{Q}
      \int\limits_{\Omega }{ \mathcal{Q} \left( \phi ,\bar{\phi } \right)dx}&=&\int\limits_{\Omega }{K\rho \left( \phi ,\bar{\phi } \right)dx} \notag\\ 
 &=&\int\limits_{\Omega }{\left( {{t}^{2}}+{{r}^{*}}^{2} \right){{\partial }_{t}}u\rho dx}+\int\limits_{\Omega }{2t{{r}^{*}}{{\partial }_{{{r}^{*}}}}u\rho dx}~.  
\end{eqnarray}
On the other hand, we have
\begin{eqnarray}
    {{\partial }_{t}}u\Tilde{\Box}u=\frac{1}{2}{{\partial }_{t}}e-{{\partial }_{r*}}\left( {{\partial }_{t}}u{{\partial }_{r*}}u \right)-\slashed{\nabla }\cdot \left( V{{\partial }_{t}}u\slashed{\nabla }u \right)~,
\end{eqnarray}
\begin{eqnarray}
    {{\partial }_{r*}}u\Tilde{\Box}u=-\frac{1}{2}{{\partial }_{r*}}e+{{\partial }_{t}}\left( {{\partial }_{t}}u{{\partial }_{r*}}u \right)-\slashed{\nabla }\cdot \left( V{{\partial }_{r*}}u\slashed{\nabla }u \right)+\frac{1}{2}{{\partial }_{r*}}V{{\left| \slashed{\nabla }u \right|}^{2}}+\frac{1}{2}V{{\partial }_{r*}}{{\left| \slashed{\nabla }u \right|}^{2}}.\notag\\
\end{eqnarray}
Thus,
\begin{eqnarray}\label{multi_K}
    \int\limits_{\Omega }{K\left( u \right)}\Tilde{\Box}udx=\int\limits_{\Omega }{\left( {{t}^{2}}+{{r}^{*}}^{2} \right){{\partial }_{t}}u}\Tilde{\Box}udx+\int\limits_{\Omega }{2t{{r}^{*}}{{\partial }_{{{r}^{*}}}}u}\Tilde{\Box}udx~,
\end{eqnarray}
where for simplicity $dx \equiv dtdr^{*}d\omega^2$. For the first term of \eqref{multi_K}, one can have  
\begin{eqnarray}\label{firsteq}
      \int\limits_{\Omega }{\left( {{t}^{2}}+{{r}^{*}}^{2} \right){{\partial }_{t}}u}\Tilde{\Box}udx&=&\frac{1}{2}\int\limits_{\Omega }{\left( {{t}^{2}}+{{r}^{*}}^{2} \right){{\partial }_{t}}edx}-\frac{1}{2}\int\limits_{\Omega }{\left( {{t}^{2}}+{{r}^{*}}^{2} \right){{\partial }_{{{r}^{*}}}}\left( {{\partial }_{t}}u{{\partial }_{{{r}^{*}}}}u \right)dx} \notag\\ 
 && -\int\limits_{\Omega }{\left( {{t}^{2}}+{{r}^{*}}^{2} \right)\slashed{\nabla }\cdot \left( V{{\partial }_{t}}u\slashed{\nabla }u \right)dx} 
\end{eqnarray}
By integrating by parts with respect to the variable $t$, we obtain 
\begin{eqnarray}
    \frac{1}{2}\int\limits_{\Omega }{\left( {{t}^{2}}+{{r}^{*}}^{2} \right){{\partial }_{t}}edx}=\frac{1}{2}\left( \int\limits_{{{\Sigma }_{{{t}_{2}}}}}{\left( {{t}^{2}}+{{r}^{*}}^{2} \right)ed\tau }-\int\limits_{{{\Sigma }_{{{t}_{1}}}}}{\left( {{t}^{2}}+{{r}^{*}}^{2} \right)ed\tau } \right)-\int\limits_{\Omega }{tedx}~,
\end{eqnarray}
where $d\tau \equiv dr^{*}d\omega^2$.
We do a similar integration by parts for the second term of \eqref{firsteq} but this time in the $r_*$, 
\begin{eqnarray}
    \frac{1}{2}\int\limits_{\Omega }{\left( {{t}^{2}}+{{r}_{*}}_{2} \right){{\partial }_{{{r}_{*}}}}\left( {{\partial }_{t}}u{{\partial }_{{{r}_{*}}}}u \right)dx}=-\int\limits_{\Omega }{{{r}_{*}}{{\partial }_{t}}u{{\partial }_{{{r}_{*}}}}udx}~.
\end{eqnarray}
The last term in  \eqref{firsteq} is zero by the divergence theorem on $S^2$. Thus, we have
\begin{eqnarray}
    \int\limits_{\Omega }{\left( {{t}^{2}}+{{r}^{*}}^{2} \right){{\partial }_{t}}u\Tilde{\Box}u}dx&=&\frac{1}{2}\left( \int\limits_{{{\Sigma }_{{{t}_{2}}}}}{\left( {{t}^{2}}+{{r}^{*}}^{2} \right)ed\tau }-\int\limits_{{{\Sigma }_{{{t}_{1}}}}}{\left( {{t}^{2}}+{{r}^{*}}^{2} \right)ed\tau } \right) \notag\\ 
 && -\int\limits_{\Omega }{tedx}-\int\limits_{\Omega }{{{r}^{*}}{{\partial }_{t}}u{{\partial }_{{{r}^{*}}}}udx}~. 
\end{eqnarray}
Next, using same technique, we derive
\begin{eqnarray}
     \int\limits_{\Omega }{2tr_*{{\partial }_{r_*}}u\tilde{\square }udx}&=&-\int\limits_{\Omega }{tr_*{{\partial }_{r_*}}edx}+\int\limits_{\Omega }{2tr_*{{\partial }_{t}}\left( {{\partial }_{t}}u{{\partial }_{r_*}}u \right)dx}\notag\\
     &&-\int\limits_{\Omega }{2tr_*\slashed{\nabla }\cdot \left( V{{\partial }_{r_*}}u\slashed{\nabla }u \right)dx} \notag\\ 
 && +\int\limits_{\Omega }{tr_*\left\{ {{\partial }_{r_*}}V{{\left| \slashed{\nabla }u \right|}^{2}}+{{\partial }_{r_*}}{{\left| \slashed{\nabla }u \right|}^{2}} \right\}dx} ~.
\end{eqnarray}
Finally, after some computation, we arrive at
\begin{eqnarray}
     && \frac{1}{2}\left( \int\limits_{{{\Sigma }_{{{t}_{2}}}}}{\left( {{t}^{2}}+{{r}_{*}}^{2}+2t{{r}_{*}}{{\partial }_{t}}u{{\partial }_{{{r}_{*}}}}ud\tau  \right)ed\tau } \right)-\frac{1}{2}\left( \int\limits_{{{\Sigma }_{t1}}}{\left( {{t}^{2}}+{{r}_{*}}^{2}+2t{{r}_{*}}{{\partial }_{t}}u{{\partial }_{{{r}_{*}}}}ud\tau  \right)ed\tau } \right) \notag\\ 
 & =&\int\limits_{\Omega }{t\left\{ r_*{{\partial }_{r_*}}V+2V \right\}{{\left| \slashed{\nabla }u \right|}^{2}}dx}+\int\limits_{\Omega }{ \mathcal{Q} \left( \phi ,\bar{\phi } \right)dx} ~.
\end{eqnarray}
As a result, we are able to obtain
\begin{eqnarray}\label{ECQ}
    E_\mathcal{C}[u,\phi](t_2)-E_\mathcal{C}[u,\phi](t_1)&=& \int\limits_{\left[ {{t}_{1}},{{t}_{2}} \right]\times \mathbb{R}\times {{S}^{2}}}{t\left\{ r_*{{\partial }_{r_*}}V+2V \right\}}{{\left| \slashed{\nabla }u \right|}^{2}}dtd{{r}_{*}}{{d}^{2}}\omega \notag\\
    &&+\int\limits_{{\left[ {{t}_{1}},{{t}_{2}} \right]\times \mathbb{R}\times {{S}^{2}}}}{\mathcal{Q}\left( \phi ,\bar{\phi } \right)dx}\notag\\
    &=&\int\limits_{\left[ {{t}_{1}},{{t}_{2}} \right]\times \mathbb{R}\times {{S}^{2}}}{ 2tV \mathcal{I} }{{\left| \slashed{\nabla }u \right|}^{2}}dtd{{r}_{*}}{{d}^{2}}\omega +\int\limits_{{\left[ {{t}_{1}},{{t}_{2}} \right]\times \mathbb{R}\times {{S}^{2}}}}{ \mathcal{Q} \left( \phi ,\bar{\phi } \right)dx}~,\notag\\
\end{eqnarray} 
where $\mathcal{I} \equiv 1+{{r}_{*}}\left( \frac{{{f}'}}{2}-\frac{f}{r} \right)$. Then, there will always be a function $\chi$ that is always greater than or equal to $2V\mathcal{I}$. This completes the proof of \eqref{lemma3}.
\end{proof}

Next, to bound \eqref{lemma3}, we need the following lemma
\begin{lemma}
Let $\gamma(u)$ be defined as a radial multiplier:
\begin{eqnarray}
    \gamma \left( u,\phi \right)\equiv g{{\partial }_{{{r}_{*}}}}u+\frac{1}{2}\left( {{\partial }_{{{r}_{*}}}}g \right)u~,
\end{eqnarray}
where $u$ is a smooth solution of \eqref{mideq} and $g$ in this paragraph denotes a function of the $t$ and $r_*$ variables only. Define the functional energy $E_{\gamma}[u,\phi]\left( t \right)$ by
    \begin{eqnarray}\label{defEgamma}
        {{E}_{\gamma }}[u,\phi]\left( t \right)\equiv\int\limits_{{{\Sigma }_{t}}}{\gamma \left( u,\phi \right){{\partial }_{t}}ud\varsigma }~,
    \end{eqnarray}
where $\Sigma_t$ denotes the spatial hypersurface at time $t$. Then, we have
\begin{eqnarray}\label{lemma4}
      2{{E}_{\gamma }}[u,\phi]\left( {{t}_{2}} \right)-2{{E}_{\gamma }}[u,\phi]\left( {{t}_{1}} \right)&=&-\int\limits_{\left[ {{t}_{1}},{{t}_{2}} \right]\times \mathbb{R}\times {{S}^{2}}}{\left( 2\left( {{\partial }_{{{r}_{*}}}}g \right) \right.}{{\left( {{\partial }_{{{r}_{*}}}}u \right)}^{2}}-\frac{1}{2}{{u}^{2}}\left( \partial _{{{r}_{*}}}^{3}g \right)-g{{\left| \slashed{\nabla }u \right|}^{2}}{{\partial }_{{{r}_{*}}}}V \notag\\ 
 &&- \left. 2\left( {{\partial }_{t}}u \right)\left( {{\partial }_{t}}g \right)\left( {{\partial }_{{{r}_{*}}}}u \right)-u\left( {{\partial }_{t}}u \right){{\partial }_{t}}{{\partial }_{{{r}_{*}}}}g \right)dx \notag\\ 
 && +\int\limits_{\left[ {{t}_{1}},{{t}_{2}} \right]\times \mathbb{R}\times {{S}^{2}}}{\left\{ g\left( {{\partial }_{{{r}_{*}}}}u \right)+\frac{1}{2}\left( {{\partial }_{{{r}_{*}}}}g \right)u \right\}\rho \left( \phi ,\bar{\phi } \right)}dx \notag\\
\end{eqnarray}
\end{lemma}
\begin{proof}
Let us consider the following equation,
    \begin{eqnarray}\label{Egamma1}
        2{{E}_{\gamma }}[u,\phi]\left( {{t}_{2}} \right)-2{{E}_{\gamma }}[u,\phi]\left( {{t}_{1}} \right)&=&2\int\limits_{\left[ {{t}_{1,}}{{t}_{2}} \right]}{{{\partial }_{t}}}{{E}_{\gamma }}[u,\phi]\left( t \right)dt \notag\\ 
 & =&2\int\limits_{\left[ {{t}_{1,}}{{t}_{2}} \right]}{{{\partial }_{t}}}\left( \int\limits_{{{\Sigma }_{t}}}{\gamma {{\partial }_{t}}ud\varsigma } \right)dt ~,
    \end{eqnarray}
where
\begin{eqnarray}\label{gammaudot}
     {{\partial }_{t}}\left( \gamma \dot{u} \right)&=&\gamma \ddot{u}+\dot{\gamma }\dot{u} \notag\\ 
 & =&\gamma {u}''+\gamma V\slashed{\Delta }u+\dot{\gamma }\dot{u}+\gamma \rho~.   
\end{eqnarray}
Each term in \eqref{gammaudot} can be expressed into
\begin{eqnarray}
    \gamma \slashed{\Delta }u=\slashed{\nabla }\cdot\left(\gamma\slashed{\nabla}u\right)-\slashed{\nabla}\gamma\cdot\slashed{\nabla}u~,
\end{eqnarray}
\begin{eqnarray}
    2\gamma {u}''=\partial_{r_*}\left(g{u'}^2+uu'g'\right)-2g'{u'}^2-uu'{g}''~,
\end{eqnarray}
\begin{eqnarray}
    2\dot{\gamma}\dot{u}=\partial_r \left(g\dot{u}^2\right)+u\dot{u}\dot{g}'+2\dot{u}\dot{g}u'~.
\end{eqnarray}
Note that we use dot notation for time derivative and prime for derivation with respect to the radial coordinate $r_*$. Combining all the terms, one can show that
\begin{eqnarray}
     2{{\partial }_{t}}\left( \gamma \dot{u} \right)&=&{{\partial }_{{{r}_{*}}}}\left( {{\left( {{u}'} \right)}^{2}}g+u{u}'{g}'-Vg{{\left| \slashed{\nabla }u \right|}^{2}}-\frac{1}{2}{{u}^{2}}{g}''+g{{{\dot{u}}}^{2}} \right) \notag\\ 
 && +2\slashed{\nabla }\cdot \left( V\gamma \slashed{\nabla }u \right)+\frac{1}{2}{{u}^{2}}{g}'''+u\dot{u}{\dot{g}}' \notag\\ 
 && +2\dot{u}\dot{g}{u}'-{{\partial }_{{{r}_{*}}}}Vg{{\left| \slashed{\nabla }u \right|}^{2}}-2{g}'{{{{u}'}}^{2}}\notag\\
 &&+\left(g{{\partial }_{{{r}_{*}}}}u+\frac{1}{2}\left( {{\partial }_{{{r}_{*}}}}g \right)u\right)\rho ~.
\end{eqnarray}
Then, \eqref{Egamma1} can be cast into
\begin{eqnarray}
       2{{E}_{\gamma }}\left( {{t}_{2}} \right)-2{{E}_{\gamma }}\left( {{t}_{1}} \right) & =&\int\limits_{\left[ {{t}_{1}},{{t}_{2}} \right]\times \mathbb{R}\times {{S}^{2}}}{{{\partial }_{{{r}_{*}}}}\left( {{\left( {{u}'} \right)}^{2}}g+u{u}'{g}'-Vg{{\left| \slashed{\nabla }u \right|}^{2}}-\frac{1}{2}{{u}^{2}}{g}''+g{{{\dot{u}}}^{2}} \right)}dx \notag\\ 
 && +\int\limits_{\left[ {{t}_{1}},{{t}_{2}} \right]\times \mathbb{R}\times {{S}^{2}}}{2\slashed{\nabla }\cdot \left( V\gamma \slashed{\nabla }u \right)}dx \notag\\ 
 && +\int\limits_{\left[ {{t}_{1}},{{t}_{2}} \right]\times \mathbb{R}\times {{S}^{2}}}{\left( \frac{1}{2}{{u}^{2}}{g}'''+u\dot{u}{\dot{g}}'+2\dot{u}\dot{g}{u}'-{{\partial }_{{{r}_{*}}}}Vg{{\left| \slashed{\nabla }u \right|}^{2}}-2{g}'{{{{u}'}}^{2}}+\gamma \rho  \right)}dx~.\notag\\
\end{eqnarray}
The first and the second term is vanishes since on the boundary as $r_*\to\pm\infty$, the smooth function $u,g$ and their derivatives go to zero.
\end{proof}

Now we are ready to derive the estimate of the uniform bound of the conformal energy for the middle component of Maxwell-Higgs system as in \eqref{prep1}, which is the most important estimate in this section. Let us set
\begin{eqnarray}
       h\left( {{r}_{*}} \right)&=&\int_{0}^{{{r}_{*}}}{\frac{1}{{{\left( 1+{{\left( \epsilon y \right)}^{2}} \right)}^{\sigma }}}}dy ~,\label{defh}\\ 
  g&=&\frac{t}{r_*}\mu h  ~,\label{defg}
\end{eqnarray}
where $\sigma \in [1,2]$, $\epsilon>0$, and $\mu=\Tilde{\mu}\left(\frac{r_*}{t}\right)$  with $\Tilde{\mu}$ a smooth function with compactly supported in $\left(-\frac{3}{4},\frac{3}{4}\right)$ which is identically 1 on $\left[ -\frac{1}{2},\frac{1}{2} \right]$. Next, using \eqref{defg}, we can re arrange the terms in \eqref{lemma4} to get
\begin{eqnarray}\label{Egammaesti}
 2{{E}_{\gamma }}\left( {{t}_{2}} \right)-2{{E}_{\gamma }}\left( {{t}_{1}} \right)&=&\frac{1}{2}\int\limits_{\left[ {{t}_{1}},{{t}_{2}} \right]\times \mathbb{R}\times {{S}^{2}}}{\left\{ {{u}^{2}}\left( \frac{t}{{{r}_{*}}}\left( {\mu }'''h+3{\mu }''{h}'+3{\mu }'{h}''+\mu {h}''' \right) \right. \right.} \notag\\ 
 && -\frac{t}{r_{*}^{2}}(3{\mu }''h+6{\mu }'{h}'+3\mu {h}'')+\left. \frac{6t\left( {\mu }'h+\mu {h}' \right)}{r_{*}^{3}}-\frac{6t\mu h}{r_{*}^{4}} \right) \notag\\ 
 && +u\dot{u}\left( -\frac{\mu h}{r_{*}^{2}}+\frac{{\mu }'h}{{{r}_{*}}}+\frac{\mu {h}'}{{{r}_{*}}}-\frac{t\mu {h}'}{r_{*}^{2}}+\frac{2t{\mu }'{h}'}{{{r}_{*}}}+\frac{t\mu {h}''}{{{r}_{*}}} \right. \notag\\ 
 &&- \left. \frac{t{\mu }'h}{r_{*}^{2}}+\frac{t{\mu }''h}{{{r}_{*}}} \right)+2\dot{u}{u}'\frac{\mu h}{{{r}_{*}}}-\frac{t\mu h}{{{r}_{*}}}{V}'{{\left| \slashed{\nabla }u \right|}^{2}} \notag\\ 
 && -\left. 2{{{{u}'}}^{2}}\left( \frac{t}{{{r}_{*}}}\left( {\mu }'h+\mu {h}' \right)-\frac{t\mu h}{r_{*}^{2}} \right)+\gamma \rho  \right\}dx  ~.
\end{eqnarray}
Lets consider the estimate for $h$ and its derivative, since $h$ is bounded function then,
\begin{eqnarray}
    \left| h \right|\le C~.
\end{eqnarray}
For the first and second derivative of $h$, for $\sigma \in [1,2]$ and $\epsilon>0$ we have
\begin{eqnarray}
    \left| {{h}'} \right|=\frac{1}{{{\left( 1+{{\left( \epsilon {{r}_{*}} \right)}^{2}} \right)}^{\sigma }}}\le \frac{1}{1+{{r}_{*}}^{2}}\le C~,
\end{eqnarray}
\begin{eqnarray}
    \left| {{h}''} \right|=\frac{2{{\varepsilon }^{2}}\sigma \left| {{r}_{*}} \right|}{{{\left( 1+{{\left( \epsilon {{r}_{*}} \right)}^{2}} \right)}^{\sigma +1}}}\le \frac{1}{1+{{r}_{*}}^{2}}\le C~.
\end{eqnarray}
For the third derivative of $h$, one can show that
\begin{eqnarray}
   \left| {{h}'''} \right|=\frac{2{{\varepsilon }^{2}}\sigma \left| 1-\left( 2\sigma +1 \right){{\left( \epsilon {{r}_{*}} \right)}^{2}} \right|}{{{\left( 1+{{\left( \epsilon {{r}_{*}} \right)}^{2}} \right)}^{\sigma +2}}}\le \frac{2{{\epsilon }^{2}}\sigma \left( \left( 2\sigma +1 \right)\left( {{\epsilon }^{2}}r_{*}^{2}+2\sigma +1 \right) \right)}{{{\left( 1+{{\left( \epsilon {{r}_{*}} \right)}^{2}} \right)}^{\sigma +2}}} ~,
\end{eqnarray}
and for $t\ge 1$, we also have
\begin{eqnarray}
    \partial _{{{r}_{*}}}^{n}\mu =\frac{1}{{{t}^{n}}}\frac{{{d}^{n}}\tilde{\mu }}{d{{\tau }^{n}}}\left( \frac{{{r}_{*}}}{t} \right)\le \frac{C}{{{t}^{n}}}\le C~.
\end{eqnarray}
Thus, we can bound each term in \eqref{Egammaesti} as follows: 
\begin{eqnarray}\label{eqb1}
    \left| \int\limits_{\left[ {{t}_{1}},{{t}_{2}} \right]\times \mathbb{R}\times {{S}^{2}}}{\left( 3\frac{t}{{{r}_{*}}}{{u}^{2}}{\mu }''{h}'+3\frac{t}{{{r}_{*}}}{{u}^{2}}{\mu }'{h}'' \right)d\varsigma } \right|\le C\int\limits_{\left[ {{t}_{1}},{{t}_{2}} \right]\times \left\{ \frac{1}{2}t\le \left| {{r}_{*}} \right|\le \frac{3}{4}t \right\}\times {{S}^{2}}}{\frac{1}{{{r}_{*}}}\frac{{{u}^{2}}}{1+{{r}_{*}}^{2}}d\varsigma }~,
\end{eqnarray}
\begin{eqnarray}
    \left| \int\limits_{\left[ {{t}_{1}},{{t}_{2}} \right]\times \mathbb{R}\times {{S}^{2}}}{6\frac{t}{{{r}_{*}}^{2}}{{u}^{2}}{\mu }'{h}'d\varsigma } \right|\le C\int\limits_{\left[ {{t}_{1}},{{t}_{2}} \right]\times \left\{ \frac{1}{2}t\le \left| {{r}_{*}} \right|\le \frac{3}{4}t \right\}\times {{S}^{2}}}{\frac{1}{{{r}_{*}}^{2}}\frac{{{u}^{2}}}{1+{{r}_{*}}^{2}}d\varsigma }~,
\end{eqnarray}
\begin{eqnarray}
    \left| \int\limits_{\left[ {{t}_{1}},{{t}_{2}} \right]\times \mathbb{R}\times {{S}^{2}}}{\left( 3\frac{t}{{{r}_{*}}^{2}}{\mu }''h \right){{u}^{2}}d\varsigma } \right|\le C\int\limits_{\left[ {{t}_{1}},{{t}_{2}} \right]\times \left\{ \frac{1}{2}t\le \left| {{r}_{*}} \right|\le \frac{3}{4}t \right\}\times {{S}^{2}}}{\frac{{{u}^{2}}}{{{r}_{*}}^{2}}d\varsigma }~,
\end{eqnarray}
\begin{eqnarray}
    \left| \int\limits_{\left[ {{t}_{1}},{{t}_{2}} \right]\times \mathbb{R}\times {{S}^{2}}}{\frac{t}{{{r}_{*}}}{{u}^{2}}{\mu }'''hd\varsigma } \right|\le C\int\limits_{\left[ {{t}_{1}},{{t}_{2}} \right]\times \left\{ \frac{1}{2}t\le \left| {{r}_{*}} \right|\le \frac{3}{4}t \right\}\times {{S}^{2}}}{\frac{{{u}^{2}}}{{{r}_{*}}}d\varsigma }~,
\end{eqnarray}
\begin{eqnarray}
    \left| \int\limits_{\left[ {{t}_{1}},{{t}_{2}} \right]\times \mathbb{R}\times {{S}^{2}}}\left(\frac{6t{\mu }'h{{u}^{2}}}{r_{*}^{3}}+\frac{6t\mu h}{r_{*}^{4}}\right)d\varsigma  \right|\le C\int\limits_{\left[ {{t}_{1}},{{t}_{2}} \right]\times \left\{ \frac{1}{2}t\le \left| {{r}_{*}} \right|\le \frac{3}{4}t \right\}\times {{S}^{2}}}{\frac{{{u}^{2}}}{r_{*}^{3}}d\varsigma }~,
\end{eqnarray}
\begin{eqnarray}
    \left| \int\limits_{\left[ {{t}_{1}},{{t}_{2}} \right]\times \mathbb{R}\times {{S}^{2}}}{\left( \frac{t}{{{r}_{*}}}\mu {h}'''+3\frac{t}{{{r}_{*}}^{2}}\mu {h}''+\frac{6t\mu {h}'}{r_{*}^{3}} \right){{u}^{2}}d\varsigma } \right|\le C\int\limits_{\left[ {{t}_{1}},{{t}_{2}} \right]\times \left\{ \frac{1}{2}t\le \left| {{r}_{*}} \right|\le \frac{3}{4}t \right\}\times {{S}^{2}}}{\frac{{{u}^{2}}}{1+{{r}_{*}}^{2}}d\varsigma }~,\notag\\
\end{eqnarray}
\begin{eqnarray}
    \left| \int\limits_{\left[ {{t}_{1}},{{t}_{2}} \right]\times \mathbb{R}\times {{S}^{2}}}{\left( \frac{\mu {h}'}{{{r}_{*}}}+\frac{t\mu {h}'}{r_{*}^{2}}+\frac{t\mu {h}''}{{{r}_{*}}} \right)u\dot{u}d\varsigma } \right|\le C\int\limits_{\left[ {{t}_{1}},{{t}_{2}} \right]\times \left\{ \frac{1}{2}t\le \left| {{r}_{*}} \right|\le \frac{3}{4}t \right\}\times {{S}^{2}}}{\frac{u\dot{u}}{1+{{r}_{*}}^{2}}d\varsigma }~,
\end{eqnarray}
\begin{eqnarray}
    \left| \int\limits_{\left[ {{t}_{1}},{{t}_{2}} \right]\times \mathbb{R}\times {{S}^{2}}}{\left( -\frac{\mu h}{r_{*}^{2}}+\frac{t{\mu }'h}{r_{*}^{2}} \right)u\dot{u}d\varsigma } \right|\le C\int\limits_{\left[ {{t}_{1}},{{t}_{2}} \right]\times \left\{ \frac{1}{2}t\le \left| {{r}_{*}} \right|\le \frac{3}{4}t \right\}\times {{S}^{2}}}{\frac{u\dot{u}}{r_{*}^{2}}d\varsigma }~,
\end{eqnarray}
\begin{eqnarray}
    \left| \int\limits_{\left[ {{t}_{1}},{{t}_{2}} \right]\times \mathbb{R}\times {{S}^{2}}}{\left( \frac{{\mu }'h}{{{r}_{*}}}+\frac{t{\mu }''h}{{{r}_{*}}} \right)u\dot{u}d\varsigma } \right|\le C\int\limits_{\left[ {{t}_{1}},{{t}_{2}} \right]\times \left\{ \frac{1}{2}t\le \left| {{r}_{*}} \right|\le \frac{3}{4}t \right\}\times {{S}^{2}}}{\frac{u\dot{u}}{{{r}_{*}}}d\varsigma }~,
\end{eqnarray}
\begin{eqnarray}
    \left| \int\limits_{\left[ {{t}_{1}},{{t}_{2}} \right]\times \mathbb{R}\times {{S}^{2}}}{\frac{2t{\mu }'{h}'}{{{r}_{*}}}u\dot{u}d\varsigma } \right|\le C\int\limits_{\left[ {{t}_{1}},{{t}_{2}} \right]\times \left\{ \frac{1}{2}t\le \left| {{r}_{*}} \right|\le \frac{3}{4}t \right\}\times {{S}^{2}}}{\frac{1}{{{r}_{*}}}\frac{u\dot{u}}{1+{{r}_{*}}^{2}}d\varsigma }~,
\end{eqnarray}
\begin{eqnarray}
    \left| \int\limits_{\left[ {{t}_{1}},{{t}_{2}} \right]\times \mathbb{R}\times {{S}^{2}}}{2\dot{u}{u}'\frac{\mu h}{{{r}_{*}}}d\varsigma } \right|\le C\int\limits_{\left[ {{t}_{1}},{{t}_{2}} \right]\times \left\{ \frac{1}{2}t\le \left| {{r}_{*}} \right|\le \frac{3}{4}t \right\}\times {{S}^{2}}}{\frac{\dot{u}{u}'}{{{r}_{*}}}d\varsigma }~,
\end{eqnarray}
\begin{eqnarray}\label{eqb12}
    \left| \int\limits_{\left[ {{t}_{1}},{{t}_{2}} \right]\times \mathbb{R}\times {{S}^{2}}}{\left( -2{{{{u}'}}^{2}}\frac{t}{{{r}_{*}}}{\mu }'h+2{{{{u}'}}^{2}}\frac{t\mu h}{r_{*}^{2}} \right)d\varsigma } \right|\le C\int\limits_{\left[ {{t}_{1}},{{t}_{2}} \right]\times \left\{ \frac{1}{2}t\le \left| {{r}_{*}} \right|\le \frac{3}{4}t \right\}\times {{S}^{2}}}{\frac{{{{{u}'}}^{2}}}{{{r}_{*}}}d\varsigma }~,
\end{eqnarray}
To get bound on \eqref{eqb1}- \eqref{eqb12}, we need the following Hardy-like estimates
\begin{lemma}
    Let $t\ge 1$, $0<\sigma$, and $\xi$ be a non negative function of $r_*$, which is positive in an open non-empty subinterval of $r_*\le \frac{1}{2}$ and $u$ be a smooth compactly supported function. Then,
\begin{eqnarray}
    \int\limits_{\left[ {{t}_{1}},{{t}_{2}} \right]\times \left\{ \left| {{r}_{*}} \right|\le \frac{1}{2} \right\}\times {{S}^{2}}}{\frac{{{u}^{2}}}{{{\left( 1+r_{*}^{2} \right)}^{\sigma +1}}}}d\varsigma \le {{\int\limits_{\left[ {{t}_{1}},{{t}_{2}} \right]\times \left\{ \left| {{r}_{*}} \right|\le \frac{1}{2} \right\}\times {{S}^{2}}}{\left( \frac{{{\left( {{\partial }_{{{r}_{*}}}}u \right)}^{2}}}{{{\left( 1+r_{*}^{2} \right)}^{\sigma }}}+\xi u \right)}}^{2}}d\varsigma~, 
\end{eqnarray}
and
\begin{eqnarray}
   \int\limits_{\left[ {{t}_{1}},{{t}_{2}} \right]\times \left\{ \left| {{r}_{*}} \right|\le \frac{3}{4} \right\}\times {{S}^{2}}}{\frac{{{u}^{2}}}{1+r_{*}^{2}}}d\varsigma \le C{{E}_{l}}[u,\phi]\left( t \right) ~.
\end{eqnarray}
\end{lemma}
\begin{proof}
The proof of the following Hardy-like estimates can be found in \cite{blue}.
\end{proof}
By employing the Hardy-like estimates, we can obtain
\begin{eqnarray}
    \int\limits_{\left[ {{t}_{1}},{{t}_{2}} \right]\times \left\{ \frac{1}{2}t\le \left| {{r}_{*}} \right|\le \frac{3}{4}t \right\}\times {{S}^{2}}}{\frac{1}{{{r}_{*}}}\frac{{{u}^{2}}}{1+{{r}_{*}}^{2}}d\varsigma }\le \int\limits_{\left[ {{t}_{1}},{{t}_{2}} \right]\times \left\{ \frac{1}{2}t\le \left| {{r}_{*}} \right|\le \frac{3}{4}t \right\}\times {{S}^{2}}}{{{\left( {{\partial }_{{{r}_{*}}}}u \right)}^{2}}d\varsigma }\le {{E}_{l}}~,
\end{eqnarray}
\begin{eqnarray}
    \int\limits_{\left[ {{t}_{1}},{{t}_{2}} \right]\times \left\{ \frac{1}{2}t\le \left| {{r}_{*}} \right|\le \frac{3}{4}t \right\}\times {{S}^{2}}}{\frac{u\dot{u}}{1+{{r}_{*}}^{2}}d\varsigma }\le \int\limits_{\left[ {{t}_{1}},{{t}_{2}} \right]\times \left\{ \frac{1}{2}t\le \left| {{r}_{*}} \right|\le \frac{3}{4}t \right\}\times {{S}^{2}}}{\frac{{{\left( u \right)}^{2}}+{{\left( {\dot{u}} \right)}^{2}}}{1+{{r}_{*}}^{2}}d\varsigma }\le {{E}_{l}}~.
\end{eqnarray}
We also get the same result for the remaining terms of \eqref{Egammaesti}.  Thus, to this point, we can rearrange \eqref{Egammaesti} in the following way
\begin{eqnarray}\label{egammaformlain}
     \int\limits_{\left[ {{t}_{1}},{{t}_{2}} \right]\times \mathbb{R}\times {{S}^{2}}}{\left(  2\frac{t}{{{r}_{*}}}\mu {h}'{{{{u}'}}^{2}}+\frac{t\mu h}{{{r}_{*}}}{V}'{{\left| \slashed{\nabla }u \right|}^{2}}  \right)}dx&\le& -2{{E}_{\gamma }}\left[ u \right]\left( {{t}_{2}} \right)+2{{E}_{\gamma }}\left[ u \right]\left( {{t}_{1}} \right) \notag\\ 
&& +\int\limits_{\left[ {{t}_{1}},{{t}_{2}} \right]\times \mathbb{R}\times {{S}^{2}}}\left\{ \frac{t}{{{r}_{*}}}\mu h{u}'+\frac{1}{2}\frac{t}{{{r}_{*}}}\left( {\mu }'h+\mu {h}' \right)u\right.\notag\\ 
&&-\left.\frac{1}{2}\frac{t\mu h}{r_{*}^{2}}u \right\}\rho dx\notag\\
 && +\hat{C}\int_{{{t}_{1}}}^{{{t}_{2}}}{{{E}_{l}}}\left[ u,\phi \right]\left( t \right)dt ~.
\end{eqnarray}
Next, for the source term, we have
\begin{eqnarray}\label{source}
       \int\limits_{\left[ {{t}_{1}},{{t}_{2}} \right]\times \mathbb{R}\times {{S}^{2}}}{\left\{ \frac{t}{{{r}_{*}}}\mu h{u}'+\frac{1}{2}\frac{t}{{{r}_{*}}}\left( {\mu }'h+\mu {h}' \right)u-\frac{1}{2}\frac{t\mu h}{r_{*}^{2}}u \right\}\rho }d\varsigma &\le& \hat{C}\int\limits_{\left[ {{t}_{1}},{{t}_{2}} \right]\times \mathbb{R}\times {{S}^{2}}}{{u}'\rho d\varsigma } \notag\\ 
 & +&
 \hat{C}\int\limits_{\left[ {{t}_{1}},{{t}_{2}} \right]\times \mathbb{R}\times {{S}^{2}}}{\frac{u\rho }{1+{{r}_{*}}^{2}}d\varsigma }\notag\\
 &+&\hat{C}\int\limits_{\left[ {{t}_{1}},{{t}_{2}} \right]\times \mathbb{R}\times {{S}^{2}}}{u\rho d\varsigma }~.\notag\\
\end{eqnarray}
For the first term of \eqref{source}, we use the Cauchy-Schwarz inequality such that
\begin{eqnarray}
   \int\limits_{\left[ {{t}_{1}},{{t}_{2}} \right]\times \left\{ \frac{1}{2}t\le \left| {{r}_{*}} \right|\le \frac{3}{4}t \right\}\times {{S}^{2}}}{{u}'\rho d\varsigma }&\le& {{\left( \int\limits_{\left[ {{t}_{1}},{{t}_{2}} \right]\times \left\{ \frac{1}{2}t\le \left| {{r}_{*}} \right|\le \frac{3}{4}t \right\}\times {{S}^{2}}}{{{\left| {{u}'} \right|}^{2}}d\varsigma } \right)}^{1/2}}\notag\\
   &&\cdot{{\left( \int\limits_{\left[ {{t}_{1}},{{t}_{2}} \right]\times \left\{ \frac{1}{2}t\le \left| {{r}_{*}} \right|\le \frac{3}{4}t \right\}\times {{S}^{2}}}{{{\left| \rho  \right|}^{2}}d\varsigma } \right)}^{1/2}} \notag\\ 
 & \le& {{E}_{l}}{{\left( \int\limits_{\left[ {{t}_{1}},{{t}_{2}} \right]\times \left\{ \frac{1}{2}t\le \left| {{r}_{*}} \right|\le \frac{3}{4}t \right\}\times {{S}^{2}}}{{{\left| \rho  \right|}^{2}}d\varsigma } \right)}^{1/2}} ~.
\end{eqnarray}
Next, let us define
\begin{eqnarray}
    \mathcal{Y}\equiv {{\left( \int\limits_{\left[ {{t}_{1}},{{t}_{2}} \right]\times \left\{ \frac{1}{2}t\le \left| {{r}_{*}} \right|\le \frac{3}{4}t \right\}\times {{S}^{2}}}{{{\left| \rho  \right|}^{2}}d\varsigma } \right)}^{1/2}}~.
\end{eqnarray}
Then, by using the Minkowski inequality, we have
\begin{eqnarray} \label{estirhoinY} \mathcal{Y}\le\mathcal{Y}^{1/2}_1+\mathcal{Y}^{1/2}_2+\mathcal{Y}^{1/2}_3+\mathcal{Y}^{1/2}_4~.
\end{eqnarray}
where
\begin{eqnarray}
     {\mathcal{Y}_{1}}&\equiv& {{ \int\limits_{\left[ {{t}_{1}},{{t}_{2}} \right]\times \left\{ \frac{1}{2}t\le \left| {{r}_{*}} \right|\le \frac{3}{4}t \right\}\times {{S}^{2}}}{{{\left| {{M}_{1}}f{{D}_{2}}\phi \bar{\phi } \right|}^{2}}d\varsigma } }}~, \\ 
  {\mathcal{Y}_{2}}&\equiv& {{ \int\limits_{\left[ {{t}_{1}},{{t}_{2}} \right]\times \left\{ \frac{1}{2}t\le \left| {{r}_{*}} \right|\le \frac{3}{4}t \right\}\times {{S}^{2}}}{{{\left| \frac{{{M}_{1}}f}{sin\theta }{{D}_{3}}\phi \bar{\phi } \right|}^{2}}d\varsigma } }}~, 
\end{eqnarray}
\begin{eqnarray}
     {\mathcal{Y}_{3}}&\equiv& {{\int\limits_{\left[ {{t}_{1}},{{t}_{2}} \right]\times \left\{ \frac{1}{2}t\le \left| {{r}_{*}} \right|\le \frac{3}{4}t \right\}\times {{S}^{2}}}{{{\left| L{{r}^{2}}{{D}_{0}}\phi \bar{\phi } \right|}^{2}}d\varsigma } }} ~,\\ 
 {\mathcal{Y}_{4}}&\equiv& {{\int\limits_{\left[ {{t}_{1}},{{t}_{2}} \right]\times \left\{ \frac{1}{2}t\le \left| {{r}_{*}} \right|\le \frac{3}{4}t \right\}\times {{S}^{2}}}{{{\left| L{{r}^{2}}{{D}_{1}}\phi \bar{\phi } \right|}^{2}}d\varsigma } }}~.
\end{eqnarray}
First, we consider the ﬁrst term of \eqref{estirhoinY} which again by Minkowski inequality satisfies
\begin{eqnarray}\label{Y1}
      {\mathcal{Y}_{1}}&\le& \int\limits_{\left[ {{t}_{1}},{{t}_{2}} \right]\times \left\{ \frac{1}{2}t\le \left| {{r}_{*}} \right|\le \frac{3}{4}t \right\}\times {{S}^{2}}}{{{\left| f\bar{\phi }{{\partial }_{2}}{{D}_{2}}\phi  \right|}^{2}}d\varsigma }+\int\limits_{\left[ {{t}_{1}},{{t}_{2}} \right]\times \left\{ \frac{1}{2}t\le \left| {{r}_{*}} \right|\le \frac{3}{4}t \right\}\times {{S}^{2}}}{{{\left| f{{D}_{2}}\phi {{\partial }_{2}}\bar{\phi } \right|}^{2}}d\varsigma } \notag\\ 
 && +\int\limits_{\left[ {{t}_{1}},{{t}_{2}} \right]\times \left\{ \frac{1}{2}t\le \left| {{r}_{*}} \right|\le \frac{3}{4}t \right\}\times {{S}^{2}}}{{{\left| f\bar{\phi }{{\partial }_{3}}{{D}_{2}}\phi  \right|}^{2}}d\varsigma }+\int\limits_{\left[ {{t}_{1}},{{t}_{2}} \right]\times \left\{ \frac{1}{2}t\le \left| {{r}_{*}} \right|\le \frac{3}{4}t \right\}\times {{S}^{2}}}{{{\left| f{{D}_{2}}\phi {{\partial }_{3}}\bar{\phi } \right|}^{2}}d\varsigma }~.\notag\\
\end{eqnarray}
Then, for the first term of \eqref{Y1}, we obtain
\begin{eqnarray}\label{y1_1}
       \int\limits_{\left[ {{t}_{1}},{{t}_{2}} \right]\times \left\{ \frac{1}{2}t\le \left| {{r}_{*}} \right|\le \frac{3}{4}t \right\}\times {{S}^{2}}}{{{\left| f\bar{\phi }{{\partial }_{2}}{{D}_{2}}\phi  \right|}^{2}}d\varsigma }&\le& \int\limits_{\left[ {{t}_{1}},{{t}_{2}} \right]\times \left\{ \frac{1}{2}t\le \left| {{r}_{*}} \right|\le \frac{3}{4}t \right\}\times {{S}^{2}}}{{{\left| f\bar{\phi }{{\partial }_{2}}{{\partial }_{2}}\phi  \right|}^{2}}d\varsigma }\notag\\
       &&+\int\limits_{\left[ {{t}_{1}},{{t}_{2}} \right]\times \left\{ \frac{1}{2}t\le \left| {{r}_{*}} \right|\le \frac{3}{4}t \right\}\times {{S}^{2}}}{{{\left| f{{\partial }_{2}}{{A}_{2}}{{\phi }^{2}} \right|}^{2}}d\varsigma } \notag\\ 
 &&+ \int\limits_{\left[ {{t}_{1}},{{t}_{2}} \right]\times \left\{ \frac{1}{2}t\le \left| {{r}_{*}} \right|\le \frac{3}{4}t \right\}\times {{S}^{2}}}{{{\left| f\bar{\phi }{{A}_{2}}{{\partial }_{2}}\phi  \right|}^{2}}d\varsigma }~.
\end{eqnarray}
Using the H\"older inequality, we then have
\begin{eqnarray}
    \int\limits_{\left[ {{t}_{1}},{{t}_{2}} \right]\times \left\{ \frac{1}{2}t\le \left| {{r}_{*}} \right|\le \frac{3}{4}t \right\}\times {{S}^{2}}}{{{\left| f\bar{\phi }{{\partial }_{2}}{{\partial }_{2}}\phi  \right|}^{2}}d\varsigma }\le {{\left\| f\phi \right\|}_{L_{loc}^{\infty }}}\int\limits_{\left[ {{t}_{1}},{{t}_{2}} \right]\times \left\{ \frac{1}{2}t\le \left| {{r}_{*}} \right|\le \frac{3}{4}t \right\}\times {{S}^{2}}}{{{\left| {{\partial }_{2}}{{\partial }_{2}}\phi  \right|}^{2}}d\varsigma }~,
\end{eqnarray}
By using equation of motion \eqref{eom2}, we can get the following estimate
\begin{eqnarray}
  \int\limits_{\left[ {{t}_{1}},{{t}_{2}} \right]\times \left\{ \frac{1}{2}t\le \left| {{r}_{*}} \right|\le \frac{3}{4}t \right\}\times {{S}^{2}}}{{{\left| {{\partial }_{2}}{{\partial }_{2}}\phi  \right|}^{2}}d\varsigma }& \le& {{\left\| \phi \right\|}_{L_{loc}^{\infty }}}\int\limits_{\left[ {{t}_{1}},{{t}_{2}} \right]\times \left\{ \frac{1}{2}t\le \left| {{r}_{*}} \right|\le \frac{3}{4}t \right\}\times {{S}^{2}}}{{{\left| \partial A \right|}^{2}}d\varsigma }\notag\\
 &&+{{\left\| A \right\|}_{L_{loc}^{\infty }}}\int\limits_{\left[ {{t}_{1}},{{t}_{2}} \right]\times \left\{ \frac{1}{2}t\le \left| {{r}_{*}} \right|\le \frac{3}{4}t \right\}\times {{S}^{2}}}{{{\left| \partial \phi  \right|}^{2}}d\varsigma } \notag\\ 
 &&+{{\left\| A \right\|}_{L_{loc}^{\infty }}}\int\limits_{\left[ {{t}_{1}},{{t}_{2}} \right]\times \left\{ \frac{1}{2}t\le \left| {{r}_{*}} \right|\le \frac{3}{4}t \right\}\times {{S}^{2}}}{{{\left| D\phi  \right|}^{2}}d\varsigma } \notag\\ 
 &\le& {{E}_{l}}\left( {{\left\| \phi \right\|}_{L_{loc}^{\infty }}}+{{\left\| A \right\|}_{L_{loc}^{\infty }}} \right)~.
\end{eqnarray}
Consequently, we derive
\begin{eqnarray}
    \int\limits_{\left[ {{t}_{1}},{{t}_{2}} \right]\times \left\{ \frac{1}{2}t\le \left| {{r}_{*}} \right|\le \frac{3}{4}t \right\}\times {{S}^{2}}}{{{\left| f\bar{\phi }{{\partial }_{2}}{{\partial }_{2}}\phi  \right|}^{2}}d\varsigma }\le {{E}_{l}}{{\left\| f\phi \right\|}_{L_{loc}^{\infty }}}\left( {{\left\| \phi \right\|}_{L_{loc}^{\infty }}}+{{\left\| A \right\|}_{L_{loc}^{\infty }}} \right)~.
\end{eqnarray}
In the same way, we can get the estimate for the remaining terms of \eqref{y1_1} as follows
\begin{eqnarray}
      \int\limits_{\left[ {{t}_{1}},{{t}_{2}} \right]\times \left\{ \frac{1}{2}t\le \left| {{r}_{*}} \right|\le \frac{3}{4}t \right\}\times {{S}^{2}}}{{{\left| f{{\partial }_{2}}{{A}_{2}}{{\phi }^{2}} \right|}^{2}}d\varsigma }&\le& {{E}_{l}}{{\left\| f\phi  \right\|}_{{{L}^{\infty }_{loc}}}}{{\left\| \phi  \right\|}_{{{L}^{\infty }_{loc}}}}~, \\ 
 \int\limits_{\left[ {{t}_{1}},{{t}_{2}} \right]\times \left\{ \frac{1}{2}t\le \left| {{r}_{*}} \right|\le \frac{3}{4}t \right\}\times {{S}^{2}}}{{{\left| f\bar{\phi }{{A}_{2}}{{\partial }_{2}}\phi  \right|}^{2}}d\varsigma }&\le& {{E}_{l}}{{\left\| f\bar{\phi } \right\|}_{{{L}^{\infty }_{loc}}}}{{\left\| A \right\|}_{{{L}^{\infty }_{loc}}}} ~.
\end{eqnarray}
Therefore, we have
\begin{eqnarray}
    \int\limits_{\left[ {{t}_{1}},{{t}_{2}} \right]\times \left\{ \frac{1}{2}t\le \left| {{r}_{*}} \right|\le \frac{3}{4}t \right\}\times {{S}^{2}}}{{{\left| f\bar{\phi }{{\partial }_{2}}{{D}_{2}}\phi  \right|}^{2}}d\varsigma }&\le& {{E}_{l}}\left( {{\left\| \phi  \right\|}_{{{L}^{\infty }_{loc}}}}+{{\left\| A \right\|}_{{{L}^{\infty }_{loc}}}}+{{\left\| f\phi  \right\|}_{{{L}^{\infty }_{loc}}}}{{\left\| \phi  \right\|}_{{{L}^{\infty }_{loc}}}}\right.\notag\\
    &&+\left.{{\left\| f\bar{\phi } \right\|}_{{{L}^{\infty }_{loc}}}}{{\left\| A \right\|}_{{{L}^{\infty }_{loc}}}} \right).\notag\\
\end{eqnarray}
Similarly, by employing some computation, we can get the estimates for the second and fourth term of \eqref{Y1}
\begin{eqnarray}
       \int\limits_{\left[ {{t}_{1}},{{t}_{2}} \right]\times \left\{ \frac{1}{2}t\le \left| {{r}_{*}} \right|\le \frac{3}{4}t \right\}\times {{S}^{2}}}{{{\left| f{{D}_{2}}\phi {{\partial }_{2}}\bar{\phi } \right|}^{2}}d\varsigma }&\le& {{\left\| f\partial \bar{\phi } \right\|}_{{{L}^{\infty }_{loc}}}}\int\limits_{\left[ {{t}_{1}},{{t}_{2}} \right]\times \left\{ \frac{1}{2}t\le \left| {{r}_{*}} \right|\le \frac{3}{4}t \right\}\times {{S}^{2}}}{{{\left| {{D}_{2}}\phi  \right|}^{2}}d\varsigma }\notag\\
       &\le& {{E}_{l}}{{\left\| f\partial \bar{\phi } \right\|}_{{{L}^{\infty }_{loc}}}}~, 
\end{eqnarray}
\begin{eqnarray}
  \int\limits_{\left[ {{t}_{1}},{{t}_{2}} \right]\times \left\{ \frac{1}{2}t\le \left| {{r}_{*}} \right|\le \frac{3}{4}t \right\}\times {{S}^{2}}}{{{\left| f{{D}_{2}}\phi {{\partial }_{3}}\bar{\phi } \right|}^{2}}d\varsigma }&\le& {{\left\| f\partial \bar{\phi } \right\|}_{{{L}^{\infty }_{loc}}}}\int\limits_{\left[ {{t}_{1}},{{t}_{2}} \right]\times \left\{ \frac{1}{2}t\le \left| {{r}_{*}} \right|\le \frac{3}{4}t \right\}\times {{S}^{2}}}{{{\left| {{D}_{2}}\phi  \right|}^{2}}d\varsigma }\notag\\
  &\le& {{E}_{l}}{{\left\| f\partial \bar{\phi } \right\|}_{{{L}^{\infty }_{loc}}}}~.
\end{eqnarray}
Notice that the estimate of the third term gives the same result as the estimate of the ﬁrst term. Thus, at this point, we get the estimate for \eqref{Y1} as follows
\begin{eqnarray}
    {\mathcal{Y}_{1}}\le c{{E}_{l}^{1/2}}\left( {{\left\| \phi  \right\|}_{{{L}^{\infty }_{loc}}}}+{{\left\| A \right\|}_{{{L}^{\infty }_{loc}}}}+{{\left\| f\phi  \right\|}_{{{L}^{\infty }_{loc}}}}{{\left\| \phi  \right\|}_{{{L}^{\infty }_{loc}}}}+{{\left\| f\bar{\phi } \right\|}_{{{L}^{\infty }_{loc}}}}{{\left\| A \right\|}_{{{L}^{\infty }_{loc}}}}+{{\left\| f\partial \bar{\phi } \right\|}_{{{L}^{\infty }_{loc}}}} \right)^{1/2}.
\end{eqnarray}
In the same way, we can get estimate for the second term $\mathcal{Y}_2$ which have the same results as $\mathcal{Y}_1$. For the third term of \eqref{estirhoinY}, by Minkowski inequality, we have
\begin{eqnarray}\label{3119}
      \int\limits_{\left[ {{t}_{1}},{{t}_{2}} \right]\times \left\{ \frac{1}{2}t\le \left| {{r}_{*}} \right|\le \frac{3}{4}t \right\}\times {{S}^{2}}}{{{\left| L{{r}^{2}}{{D}_{0}}\phi \bar{\phi } \right|}^{2}}d\varsigma }&\le& \int\limits_{\left[ {{t}_{1}},{{t}_{2}} \right]\times \left\{ \frac{1}{2}t\le \left| {{r}_{*}} \right|\le \frac{3}{4}t \right\}\times {{S}^{2}}}{{{\left| {{\partial }_{t}}\left( {{r}^{2}}{{D}_{0}}\phi \bar{\phi } \right) \right|}^{2}}d\varsigma } \notag\\ 
 && +\int\limits_{\left[ {{t}_{1}},{{t}_{2}} \right]\times \left\{ \frac{1}{2}t\le \left| {{r}_{*}} \right|\le \frac{3}{4}t \right\}\times {{S}^{2}}}{{{\left| {{\partial }_{{{r}_{*}}}}\left( {{r}^{2}}{{D}_{0}}\phi \bar{\phi } \right) \right|}^{2}}d\varsigma }~.\notag\\
\end{eqnarray}
We can estimate \eqref{3119} in the following way
\begin{eqnarray}
      \int\limits_{\left[ {{t}_{1}},{{t}_{2}} \right]\times \left\{ \frac{1}{2}t\le \left| {{r}_{*}} \right|\le \frac{3}{4}t \right\}\times {{S}^{2}}}{{{\left| {{\partial }_{t}}\left( {{r}^{2}}{{D}_{0}}\phi \bar{\phi } \right) \right|}^{2}}d\varsigma }&\le& \int\limits_{\left[ {{t}_{1}},{{t}_{2}} \right]\times \left\{ \frac{1}{2}t\le \left| {{r}_{*}} \right|\le \frac{3}{4}t \right\}\times {{S}^{2}}}{{{\left| {{r}^{2}}\left( {{\partial }_{0}}{{D}_{0}}\phi  \right)\bar{\phi } \right|}^{2}}d\varsigma } \notag\\ 
 &&+\int\limits_{\left[ {{t}_{1}},{{t}_{2}} \right]\times \left\{ \frac{1}{2}t\le \left| {{r}_{*}} \right|\le \frac{3}{4}t \right\}\times {{S}^{2}}}{{{\left| {{r}^{2}}{{D}_{0}}\phi {{\partial }_{0}}\bar{\phi } \right|}^{2}}d\varsigma } \notag\\ 
 & \le& {{\left\| {{r}^{2}}\bar{\phi } \right\|}_{{{L}^{\infty }_{loc}}}}\int\limits_{\left[ {{t}_{1}},{{t}_{2}} \right]\times \left\{ \frac{1}{2}t\le \left| {{r}_{*}} \right|\le \frac{3}{4}t \right\}\times {{S}^{2}}}{{{\left| \left( {{\partial }_{0}}{{D}_{0}}\phi  \right) \right|}^{2}}d\varsigma } \notag\\ 
 & \le& {{E}_{l}}\left( {{\left\| \phi  \right\|}_{{{L}^{\infty }_{loc}}}}{{\left\| {{r}^{2}}\bar{\phi } \right\|}_{{{L}^{\infty }_{loc}}}}+{{\left\| A \right\|}_{{{L}^{\infty }_{loc}}}}{{\left\| {{r}^{2}}\bar{\phi } \right\|}_{{{L}^{\infty }_{loc}}}} \right), \notag\\ 
\end{eqnarray}
and
\begin{eqnarray}
     \int\limits_{\left[ {{t}_{1}},{{t}_{2}} \right]\times \left\{ \frac{1}{2}t\le \left| {{r}_{*}} \right|\le \frac{3}{4}t \right\}\times {{S}^{2}}}{{{\left| {{\partial }_{{{r}_{*}}}}\left( {{r}^{2}}{{D}_{0}}\phi \bar{\phi } \right) \right|}^{2}}d\varsigma }&=&\int\limits_{\left[ {{t}_{1}},{{t}_{2}} \right]\times \left\{ \frac{1}{2}t\le \left| {{r}_{*}} \right|\le \frac{3}{4}t \right\}\times {{S}^{2}}}{{{\left| \left( {{r}^{2}}{{D}_{0}}\phi \bar{\phi } \right) \right|}^{2}}d{{\omega }^{2}}} \notag\\ 
 & \le& {{\left\| {{r}^{2}}\phi  \right\|}_{L_{loc}^{\infty }}}\int\limits_{\left[ {{t}_{1}},{{t}_{2}} \right]\times \left\{ \frac{1}{2}t\le \left| {{r}_{*}} \right|\le \frac{3}{4}t \right\}\times {{S}^{2}}}{{{\left| {{D}_{0}}\phi  \right|}^{2}}d{{\omega }^{2}}} \notag\\ 
 & \le& {{E}_{l}}{{\left\| {{r}^{2}}\phi  \right\|}_{L_{loc}^{\infty }}} . 
\end{eqnarray}
Consequently, the estimate for the third term in equation \eqref{estirhoinY} is derived as follows:
\begin{eqnarray}\label{y3}
    \mathcal{Y}_3\le {{E}_{l}^{1/2}}\left( {{\left\| \phi  \right\|}_{{{L}^{\infty }_{loc}}}}{{\left\| {{r}^{2}}\bar{\phi } \right\|}_{{{L}^{\infty }_{loc}}}}+{{\left\| A \right\|}_{{{L}^{\infty }_{loc}}}}{{\left\| {{r}^{2}}\bar{\phi } \right\|}_{{{L}^{\infty }_{loc}}}}+1 \right)^{1/2}.
\end{eqnarray}
It should be noted that the fourth term yields the same result as in equation \eqref{y3}. Thus, the final estimate for equation \eqref{estirhoinY} is given by
\begin{eqnarray}\label{Yfinal}
\mathcal{Y}&\le& c{{E}_{l}^{1/2}}\left( {{\left\| \phi  \right\|}_{L_{loc}^{\infty }}}+{{\left\| A \right\|}_{L_{loc}^{\infty }}}+{{\left\| f\phi  \right\|}_{L_{loc}^{\infty }}}{{\left\| \phi  \right\|}_{L_{loc}^{\infty }}}+{{\left\| f\bar{\phi } \right\|}_{L_{loc}^{\infty }}}{{\left\| A \right\|}_{L_{loc}^{\infty }}} \right.\notag\\
&&+\left. {{\left\| f\partial \bar{\phi } \right\|}_{L_{loc}^{\infty }}}+{{\left\| \phi  \right\|}_{L_{loc}^{\infty }}}{{\left\| {{r}^{2}}\bar{\phi } \right\|}_{L_{loc}^{\infty }}}+{{\left\| A \right\|}_{L_{loc}^{\infty }}}{{\left\| {{r}^{2}}\bar{\phi } \right\|}_{L_{loc}^{\infty }}}+1 \right)^{1/2}~.
\end{eqnarray}
Thus, the estimate for the first term of \eqref{source} become
\begin{eqnarray}\label{hasilfirstsource}
    \int\limits_{\left[ {{t}_{1}},{{t}_{2}} \right]\times \left\{ \frac{1}{2}t\le \left| {{r}_{*}} \right|\le \frac{3}{4}t \right\}\times {{S}^{2}}}{t{u}'\rho d\varsigma }&\le& {{E}_{l}}\left( {{\left\| \phi  \right\|}_{L_{loc}^{\infty }}}+{{\left\| A \right\|}_{L_{loc}^{\infty }}}+{{\left\| f\phi  \right\|}_{L_{loc}^{\infty }}}{{\left\| \phi  \right\|}_{L_{loc}^{\infty }}}+{{\left\| f\bar{\phi } \right\|}_{L_{loc}^{\infty }}}{{\left\| A \right\|}_{L_{loc}^{\infty }}} \right.\notag\\
    &&+\left. {{\left\| f\partial \bar{\phi } \right\|}_{L_{loc}^{\infty }}}+{{\left\| \phi  \right\|}_{L_{loc}^{\infty }}}{{\left\| {{r}^{2}}\bar{\phi } \right\|}_{L_{loc}^{\infty }}}+{{\left\| A \right\|}_{L_{loc}^{\infty }}}{{\left\| {{r}^{2}}\bar{\phi } \right\|}_{L_{loc}^{\infty }}}+1 \right)^{1/2}.\notag\\
\end{eqnarray}
For the second term of \eqref{source}, by using the Cauchy-Schwarz inequality along with Hardy-like estimate, we obatin
\begin{eqnarray}
       \int\limits_{\left[ {{t}_{1}},{{t}_{2}} \right]\times \mathbb{R}\times {{S}^{2}}}{\frac{u\rho }{1+{{r}_{*}}^{2}}d\varsigma }&\le& {{\left( \int\limits_{\left[ {{t}_{1}},{{t}_{2}} \right]\times \mathbb{R}\times {{S}^{2}}}{\frac{{{\left| u \right|}^{2}}}{{{\left( 1+{{r}_{*}}^{2} \right)}^{2}}}d\varsigma } \right)}^{1/2}}{{\left( \int\limits_{\left[ {{t}_{1}},{{t}_{2}} \right]\times \mathbb{R}\times {{S}^{2}}}{\rho d\varsigma } \right)}^{1/2}} \notag\\ 
 & \le& {{E}_{l}^{1/2}}{{\left( \int\limits_{\left[ {{t}_{1}},{{t}_{2}} \right]\times \mathbb{R}\times {{S}^{2}}}{\rho d\varsigma } \right)}^{1/2}}~.
\end{eqnarray}
Thus, considering \eqref{Yfinal}, we get the same result as in \eqref{hasilfirstsource},
\begin{eqnarray}
    \int\limits_{\left[ {{t}_{1}},{{t}_{2}} \right]\times \mathbb{R}\times {{S}^{2}}}{\frac{u\rho }{1+{{r}_{*}}^{2}}d\varsigma }&\le& {{E}_{l}}\left( {{\left\| \phi  \right\|}_{L_{loc}^{\infty }}}+{{\left\| A \right\|}_{L_{loc}^{\infty }}}+{{\left\| f\phi  \right\|}_{L_{loc}^{\infty }}}{{\left\| \phi  \right\|}_{L_{loc}^{\infty }}}+{{\left\| f\bar{\phi } \right\|}_{L_{loc}^{\infty }}}{{\left\| A \right\|}_{L_{loc}^{\infty }}} \right.\notag\\
    &&+\left. {{\left\| f\partial \bar{\phi } \right\|}_{L_{loc}^{\infty }}}+{{\left\| \phi  \right\|}_{L_{loc}^{\infty }}}{{\left\| {{r}^{2}}\bar{\phi } \right\|}_{L_{loc}^{\infty }}}+{{\left\| A \right\|}_{L_{loc}^{\infty }}}{{\left\| {{r}^{2}}\bar{\phi } \right\|}_{L_{loc}^{\infty }}}+1 \right)^{1/2}.\notag\\
\end{eqnarray}
For the third term of \eqref{source}, we obtain
\begin{eqnarray}
       \int\limits_{\left[ {{t}_{1}},{{t}_{2}} \right]\times \mathbb{R}\times {{S}^{2}}}{u\rho d\varsigma }&\le& {{\left( \int\limits_{\left[ {{t}_{1}},{{t}_{2}} \right]\times \mathbb{R}\times {{S}^{2}}}{{{\left| u \right|}^{2}}d\varsigma } \right)}^{1/2}}{{\left( \int\limits_{\left[ {{t}_{1}},{{t}_{2}} \right]\times \mathbb{R}\times {{S}^{2}}}{{{\left| \rho  \right|}^{2}}d\varsigma } \right)}^{1/2}} \notag\\ 
 & \le& E_{l}^{1/2}{{\left( \int\limits_{\left[ {{t}_{1}},{{t}_{2}} \right]\times \mathbb{R}\times {{S}^{2}}}{{{\left| \rho  \right|}^{2}}d\varsigma } \right)}^{1/2}}\notag\\
 &\le& {{E}_{l}}\left( {{\left\| \phi  \right\|}_{L_{loc}^{\infty }}}+{{\left\| A \right\|}_{L_{loc}^{\infty }}}+{{\left\| f\phi  \right\|}_{L_{loc}^{\infty }}}{{\left\| \phi  \right\|}_{L_{loc}^{\infty }}}+{{\left\| f\bar{\phi } \right\|}_{L_{loc}^{\infty }}}{{\left\| A \right\|}_{L_{loc}^{\infty }}} \right.\notag\\
    &&+\left. {{\left\| f\partial \bar{\phi } \right\|}_{L_{loc}^{\infty }}}+{{\left\| \phi  \right\|}_{L_{loc}^{\infty }}}{{\left\| {{r}^{2}}\bar{\phi } \right\|}_{L_{loc}^{\infty }}}+{{\left\| A \right\|}_{L_{loc}^{\infty }}}{{\left\| {{r}^{2}}\bar{\phi } \right\|}_{L_{loc}^{\infty }}}+1 \right)^{1/2}.\notag\\
\end{eqnarray}
Thus, using \eqref{AH4} and \eqref{phiH4}, \eqref{source} become
\begin{eqnarray}
     \int\limits_{\left[ {{t}_{1}},{{t}_{2}} \right]\times \mathbb{R}\times {{S}^{2}}}{\left\{ \frac{t}{{{r}_{*}}}\mu h{u}'+\frac{1}{2}\frac{t}{{{r}_{*}}}\left( {\mu }'h+\mu {h}' \right)u-\frac{1}{2}\frac{t\mu h}{r_{*}^{2}}u \right\}\rho }d\varsigma\le\bar{C}\left(1+t\right) {{E}_{l}}~.
\end{eqnarray}
Therefore, eq \eqref{egammaformlain} becomes
\begin{eqnarray}
        \int\limits_{\left[ {{t}_{1}},{{t}_{2}} \right]\times \mathbb{R}\times {{S}^{2}}}{\left(  2\frac{t}{{{r}_{*}}}\mu {h}'{{{{u}'}}^{2}}+\frac{t\mu h}{{{r}_{*}}}{V}'{{\left| \slashed{\nabla }u \right|}^{2}}  \right)}dx&\le& -2{{E}_{\gamma }}\left[ u \right]\left( {{t}_{2}} \right)+2{{E}_{\gamma }}\left[ u \right]\left( {{t}_{1}} \right) \notag\\
 &&+\bar{C}\int_{{{t}_{1}}}^{{{t}_{2}}}\left(1+t\right){{E}_{l}}\left( t \right)dt~.
\end{eqnarray}
The same result is obtained also for
\begin{eqnarray}\label{eq1}
        \int\limits_{\left[ {{t}_{1}},{{t}_{2}} \right]\times \mathbb{R}\times {{S}^{2}}}{\left(  2\frac{t}{{{r}_{*}}}\mu {h}'{{{\slashed{\nabla}{u}'}}^{2}}+\frac{t\mu h}{{{r}_{*}}}{V}'{{\left| \slashed{\Delta }u \right|}^{2}}  \right)}dx&\le& -2{{E}_{\gamma }}\left[ \slashed{\nabla}u \right]\left( {{t}_{2}} \right)+2{{E}_{\gamma }}\left[ \slashed{\nabla}u \right]\left( {{t}_{1}} \right) \notag\\
 &&+\bar{C}\int_{{{t}_{1}}}^{{{t}_{2}}}\left(1+t\right){{E}_{l}\left[ \slashed{\nabla}u \right]}\left( t \right)dt~.\notag\\
\end{eqnarray}
From Hardy-like inequality we have
\begin{eqnarray}
    \int\limits_{\left\{ t \right\}\times \left\{ \left| {{r}_{*}} \right|\le \frac{1}{2}t \right\}\times {{S}^{2}}}{t\frac{\left| \slashed{\nabla }u \right|}{{{\left( 1+{{\left( \varepsilon {{r}_{*}} \right)}^{2}} \right)}^{\sigma +1}}}}d\varsigma \le \int\limits_{\left[ {{t}_{1}},{{t}_{2}} \right]\times \mathbb{R}\times {{S}^{2}}}{\left(  2\frac{t}{{{r}_{*}}}\mu {h}'{{{\slashed{\nabla}{u}'}}^{2}}+\frac{t\mu h}{{{r}_{*}}}{V}'{{\left| \slashed{\Delta }u \right|}^{2}}  \right)}d\varsigma~.\notag\\
\end{eqnarray}
Let $a\ge 1$ such that $\sup \left( {{\chi }_{trap}} \right)\subseteq \left[ -a/2,a/2 \right]$, then for $C>0$ and ${{r}_{*}}\in \left[ -a/2,a/2 \right]$
\begin{eqnarray}
    {{\chi }_{trap}}\le \frac{C}{{{\left( 1+{{\left( \varepsilon {{r}_{*}} \right)}^{2}} \right)}^{\sigma +1}}}~,
\end{eqnarray}
then
\begin{eqnarray}
    \int\limits_{\left\{ t \right\}\times \left\{ \left| {{r}_{*}} \right|\le \frac{1}{2}t \right\}\times {{S}^{2}}}{t{{\chi }_{trap}}{{\left| \slashed{\nabla }u \right|}^{2}}}d\varsigma \le C\int\limits_{\left\{ t \right\}\times \left\{ \left| {{r}_{*}} \right|\le \frac{1}{2}t \right\}\times {{S}^{2}}}{t\frac{\left|  \slashed{\nabla }u \right|}{{{\left( 1+{{\left( \varepsilon {{r}_{*}} \right)}^{2}} \right)}^{\sigma +1}}}}d\varsigma~.
\end{eqnarray}
Integrating from $a$ to $t'\ge a$, then using \eqref{eq1}, we get
\begin{eqnarray}
       \int\limits_{\left[ a,t' \right]\times \left\{ \left| {{r}_{*}} \right|\le \frac{1}{2}t \right\}\times {{S}^{2}}}{t{{\chi }_{trap}}{{\left| \slashed{\nabla }u \right|}^{2}}}dx&\le&\bar{C}\left| -2\left. {{E}_{\gamma }}\left[ \slashed{\nabla }u \right]\left( t \right) \right|_{a}^{t'} \right| \notag\\ 
 & +&\bar{C}\int_{a}^{t'}\left(1+t\right){{E}_{l}}\left[ \slashed{\nabla }u \right]\left( t \right)dt.
\end{eqnarray}
By \eqref{lemma3}, we obtain
\begin{eqnarray}\label{3135}
      {{E}_{\mathcal{C}}}(t')\le {{E}_{\mathcal{C}}}(a)+\bar{C}\left| -2\left. {{E}_{\gamma }}\left[ \slashed{\nabla }u \right]\left( t \right) \right|_{a}^{t'} \right|+\bar{C}\int_{a}^{t'}(1+t){{E}_{l}}\left[ \slashed{\nabla }u \right]\left( t \right)dt .
\end{eqnarray}


Now, let consider the following lemma
\begin{lemma}
    Let $E_\gamma$ as in \eqref{defEgamma}, then we can attain
    \begin{eqnarray}\label{egammaesti}
        E_\gamma\le \hat{C} \left(1+t\right) E_l~.
    \end{eqnarray}
\end{lemma}
\begin{proof}
    Using \eqref{defg}, we can write down \eqref{defEgamma} in the following way
    \begin{eqnarray}\label{egammaestikeEL}
          {{E}_{\gamma }}[u,\phi ]\left( t \right)=\int\limits_{{{\Sigma }_{t}}}{\left( \frac{t}{{{r}_{*}}}\mu h{u}'+\frac{1}{2}\frac{t}{{{r}_{*}}}\left( {\mu }'h+\mu {h}' \right)u-\frac{1}{2}\frac{t\mu h}{r_{*}^{2}}u \right){{\partial }_{t}}ud\varsigma } ~.
    \end{eqnarray}
Then we can estimate each term in \eqref{egammaestikeEL} as follows,
\begin{eqnarray}
       \int\limits_{{{\Sigma }_{t}}}{\frac{t}{{{r}_{*}}}\mu h{u}'{{\partial }_{t}}ud\varsigma }&\le& \hat{C}\int\limits_{{{\Sigma }_{t}}}{{u}'{{\partial }_{t}}ud\varsigma } \notag\\ 
 & \le& \hat{C}{{\left( \int\limits_{{{\Sigma }_{t}}}{{{\left| {{\partial }_{t}}u \right|}^{2}}d\varsigma } \right)}^{1/2}}{{\left( \int\limits_{{{\Sigma }_{t}}}{{{\left| {{u}'} \right|}^{2}}d\varsigma } \right)}^{1/2}} \notag\\ 
 & \le& \hat{C}{{E}_{l}} ~,
\end{eqnarray}
\begin{eqnarray}
       \frac{1}{2}\left( \int\limits_{{{\Sigma }_{t}}}{\frac{t}{{{r}_{*}}}\left( {\mu }'h+\mu {h}' \right)u{{\partial }_{t}}ud\varsigma } \right)&\le& \hat{C}\left( \int\limits_{{{\Sigma }_{t}}}{\frac{u}{{{r}_{*}}}{{\partial }_{t}}ud\varsigma }+\int\limits_{{{\Sigma }_{t}}}{\frac{u}{1+{{r}_{*}}^{2}}{{\partial }_{t}}ud\varsigma } \right) \notag\\ 
 & \le& \hat{C}\left( {{\left( \int\limits_{{{\Sigma }_{t}}}{\frac{{{\left| u \right|}^{2}}}{{{r}_{*}}^{2}}d\varsigma } \right)}^{1/2}}{{\left( \int\limits_{{{\Sigma }_{t}}}{{{\left| {{\partial }_{t}}u \right|}^{2}}d\varsigma } \right)}^{1/2}} \right. \notag\\ 
 && +\left. {{\left( \int\limits_{{{\Sigma }_{t}}}{\frac{{{\left| u \right|}^{2}}}{{{\left( 1+{{r}_{*}}^{2} \right)}^{2}}}d\varsigma } \right)}^{1/2}}{{\left( \int\limits_{{{\Sigma }_{t}}}{{{\left| {{\partial }_{t}}u \right|}^{2}}d\varsigma } \right)}^{1/2}} \right) \notag\\ 
 & \le& \hat{C}{{E}_{l}} ~,
\end{eqnarray}
and
\begin{eqnarray}
     -\frac{1}{2}\int\limits_{{{\Sigma }_{t}}}{\left( \frac{t\mu h}{r_{*}^{2}}u \right){{\partial }_{t}}ud\varsigma }&\le& \hat{C}\int\limits_{{{\Sigma }_{t}}}{\left( u \right){{\partial }_{t}}ud\varsigma } \notag\\ 
 & \le &\hat{C}t{{E}_{l}} ~.
\end{eqnarray}
Hence, combining all the term we can attain \eqref{egammaestikeEL}.
\end{proof}

Now using \eqref{egammaesti}, \eqref{3135} becomes
\begin{eqnarray}\label{Ec_acent}
      {{E}_{\mathcal{C}}}(t')&\le& {{E}_{\mathcal{C}}}(a)+\bar{C}\underset{t\in \left[ 0,t' \right]}{\mathop{\sup }}\,\left( \left( 1+t \right){{E}_{l}}\left[ \slashed{\nabla }u \right]\left( t \right) \right)+\bar{C}\int_{a}^{t'}(1+t){{E}_{l}}\left[ \slashed{\nabla }u \right]\left( t \right)dt.\notag\\ 
\end{eqnarray}
where $\bar{C}$ is a constant depending solely on the initial data. To get bound on $E_{\mathcal{C}}(a)$, we need to consider for $0\le t_1\le 1$ cases. Since $V>0$ and $\chi_{trap}$ has a compact support then $\chi_{trap}\le CV$, thus
\begin{eqnarray}
   \int_{[0,{{t}_{1}}]\times \mathbb{R}\times {{\mathbb{S}}^{2}}}t{{{\chi}_{\text{trap}}}}|\nabla u{{|}^{2}}dx&\le& C\int_{[0,{{t}_{1}}]\times \mathbb{R}\times {{\mathbb{S}}^{2}}}{V}|\slashed{\nabla} u{{|}^{2}}dx \notag\\ 
 & \le& C\int_{0}^{{{t}_{1}}}{E}[u](t)dt \notag\\ 
 & \le& C\int_{0}^{{{t}_{1}}}{{{E}_{\mathcal{C}}}}[u](t)dt.
\end{eqnarray}
and for $t_2\ge 1$ we have
\begin{eqnarray}
   \int_{[{{t}_{1}},{{t}_{2}}]\times \mathbb{R}\times {{\mathbb{S}}^{2}}}t{{{\chi}_{\text{trap}}}}|\slashed{\nabla}u{{|}^{2}}dx&\le& C\int_{[{{t}_{1}},{{t}_{2}}]\times \mathbb{R}\times {{\mathbb{S}}^{2}}}t^2{V}|\slashed{\nabla} u{{|}^{2}}dx \notag\\ 
 & \le& C\int_{{{t}_{1}}}^{{{t}_{2}}}{{{E}_{\mathcal{C}}}}[u](t)dt.
\end{eqnarray}
Now we need to estimate the second term of \eqref{lemma3},
\begin{eqnarray}
\int\limits_{[{{t}_{1}},{{t}_{2}}]\times \mathbb{R}\times {{\mathbb{S}}^{2}}}{\mathcal{Q}\left( \phi ,\bar{\phi } \right)dx}&=&\int\limits_{[{{t}_{1}},{{t}_{2}}]\times \mathbb{R}\times {{\mathbb{S}}^{2}}}{\left( {{t}^{2}}+{{r}_{*}}^{2} \right){{\partial }_{t}}u\rho dx}+\int\limits_{[{{t}_{1}},{{t}_{2}}]\times \mathbb{R}\times {{\mathbb{S}}^{2}}}{2t{{r}_{*}}{{\partial }_{{{r}_{*}}}}u\rho dx} \notag\\ 
& \le& C\int\limits_{{{t}_{1}}}^{{{t}_{2}}}{{{t}^{2}}{{\left( \int\limits_{[{{t}_{1}},{{t}_{2}}]\times \mathbb{R}\times {{\mathbb{S}}^{2}}}{{{\left| {{\partial }_{t}}u \right|}^{2}}d\varsigma } \right)}^{1/2}}{{\left( \int\limits_{[{{t}_{1}},{{t}_{2}}]\times \mathbb{R}\times {{\mathbb{S}}^{2}}}{{{\left| \rho  \right|}^{2}}d\varsigma } \right)}^{1/2}}}dt \notag\\ 
& &+C\int\limits_{{{t}_{1}}}^{{{t}_{2}}}{{{t}^{2}}{{\left( \int\limits_{[{{t}_{1}},{{t}_{2}}]\times \mathbb{R}\times {{\mathbb{S}}^{2}}}{{{\left| {{\partial }_{{{r}_{*}}}}u \right|}^{2}}d\varsigma } \right)}^{1/2}}{{\left( \int\limits_{[{{t}_{1}},{{t}_{2}}]\times \mathbb{R}\times {{\mathbb{S}}^{2}}}{{{\left| \rho  \right|}^{2}}d\varsigma } \right)}^{1/2}}}dt \notag\\ 
& \le&\bar{C}\int\limits_{{{t}_{1}}}^{{{t}_{2}}}{{{t}^{2}}}(1+t) E\left[ u \right]\left( t \right)dt 
\end{eqnarray}
Therefore for all $t_2,t_1\ge 0$, by \eqref{lemma3}, we obtain
\begin{eqnarray}\label{EQ}
   E_{\mathcal{C}}(t_2)- E_{\mathcal{C}}(t_1)&\le&\int_{[{{t}_{1}},{{t}_{2}}]\times \mathbb{R}\times {{\mathbb{S}}^{2}}}t{{{\chi}_{\text{trap}}}}|\slashed{\nabla}u{{|}^{2}}dx+\int\limits_{{\left[ {{t}_{1}},{{t}_{2}} \right]\times \mathbb{R}\times {{S}^{2}}}}{ \mathcal{Q} \left( \phi ,\bar{\phi } \right)dx}\notag\\
   &\le& \bar{C}\int_{{{t}_{1}}}^{{{t}_{2}}}\left(t^3+t^2\right){{{E}_{\mathcal{C}}}}[u](t)dt.
\end{eqnarray}
By employing the Gr\"onwall lemma, we have
\begin{eqnarray}\label{expEC}
     {{E}_{\mathcal{C}}}({{t}_{2}})\le {{E}_{\mathcal{C}}}({{t}_{1}})e^{\left( \bar{C} \left[ (({{t}_{2}}-{{t}_{1}})^4+({{t}_{2}}-{{t}_{1}})^3 \right] \right)} .
\end{eqnarray}
Now taking $t_1=0$ and $t_2=a$ we get bound on the conformal energy
\begin{eqnarray}
    E_{\mathcal{C}}(a)\le C E_{\mathcal{C}}(0)
\end{eqnarray}
 Hence, for $t'\ge a$
\begin{eqnarray}\label{Ec_acent2}
      {{E}_{\mathcal{C}}}(t')\le \bar{C}\left({{E}_{\mathcal{C}}}(0)+\underset{t\in \left[ 0,t' \right]}{\mathop{\sup }}\,\left( \left( 1+t \right){{E}_{l}}\left[ \slashed{\nabla }u \right]\left( t \right) \right)+\int_{0}^{t'}(1+t){{E}_{l}}\left[ \slashed{\nabla }u \right]\left( t \right)dt\right).
\end{eqnarray}
Since ${{E}_{l}}\left[\slashed{\nabla }u \right]\left( t \right)\le CE\left[ \slashed{\nabla }u \right]\left( t \right)$ and by lemma \ref{lema2}, $E\left[ \slashed{\nabla }u \right]\left( t \right)=E\left[ \slashed{\nabla }u \right]\left( 0 \right)$. Then, for $t\ge a$, we obtain
\begin{eqnarray}\label{EC1}
    {{E}_{\mathcal{C}}}(t)\le \bar{C}\left( 1+t+t^2\right)\left( {{E}_{\mathcal{C}}}(0)+{E}\left[ \slashed{\nabla }u \right]\left( 0 \right) \right)~.
\end{eqnarray}
Since by Holder inequality
\begin{eqnarray}\label{holderEl}
    {{E}_{l}}\left[ \slashed{\nabla }u \right]\left( t \right)\le {{\left\{ {{E}_{l}}\left[ \slashed{\Delta }^2u \right]\left( t \right) \right\}}^{1/4}}{{\left\{ {{E}_{l}}\left[ u \right]\left( t \right) \right\}}^{3/4}}~,
\end{eqnarray}
and for $|r_*|<\frac{3t}{4}$,
\begin{eqnarray}
       \frac{{{e}_{C}}}{{{t}^{2}}}&=&\frac{1}{4}\left[ \frac{{{\left( t+{{r}_{*}} \right)}^{2}}}{{{t}^{2}}}{{\left( \dot{u}+{u}' \right)}^{2}} \right]+\frac{1}{4}\left[ \frac{{{\left( t-{{r}_{*}} \right)}^{2}}}{{{t}^{2}}}{{\left( \dot{u}-{u}' \right)}^{2}} \right]+\frac{1}{2}\left( 1+\frac{{{r}_{*}}^{2}}{{{t}^{2}}} \right)V{{\left| \slashed{\nabla }u \right|}^{2}}+\frac{e}{{{t}^{2}}} \notag\\ 
 & \ge& \frac{1}{64}\left[ {{\left( \dot{u}+{u}' \right)}^{2}} \right]+\frac{1}{64}\left[ {{\left( \dot{u}-{u}' \right)}^{2}} \right]+\frac{1}{2}\left( 1+\frac{{{r}_{*}}^{2}}{{{t}^{2}}} \right)V{{\left| \slashed{\nabla }u \right|}^{2}}+\frac{e}{{{t}^{2}}} \notag\\ 
 & =&\frac{1}{32}\left( {{{\dot{u}}}^{2}}+{{{{u}'}}^{2}} \right)+\frac{1}{32}V{{\left| \slashed{\nabla }u \right|}^{2}}+\left( \frac{15}{32}+\frac{{{r}_{*}}^{2}}{2{{t}^{2}}} \right)V{{\left| \slashed{\nabla }u \right|}^{2}}+\frac{e}{{{t}^{2}}} \notag\\ 
 & =&\frac{e}{32}-\frac{1}{32}\left( {{\left| {{D}_{0}}\phi  \right|}^{2}}+{{\left| {{D}_{i}}\phi  \right|}^{2}} \right)+\left( \frac{15}{32}+\frac{{{r}_{*}}^{2}}{2{{t}^{2}}} \right)V{{\left| \slashed{\nabla }u \right|}^{2}}+\frac{e}{{{t}^{2}}} ~.
\end{eqnarray}
We have
\begin{eqnarray}
    \frac{{{e}_{C}}}{{{t}^{2}}}+\frac{1}{32}\left( {{\left| {{D}_{0}}\phi  \right|}^{2}}+{{\left| {{D}_{i}}\phi  \right|}^{2}} \right)\ge \frac{e}{32}+\left( \frac{15}{32}+\frac{{{r}_{*}}^{2}}{2{{t}^{2}}} \right)V{{\left| \slashed{\nabla }u \right|}^{2}}+\frac{e}{{{t}^{2}}}~.
\end{eqnarray}
Therefore
\begin{eqnarray}\label{EL1}
    E_l[u]\le C \left(\frac{E_{\mathcal{C}}}{t^2}+E(0)\right)~.
\end{eqnarray}
Subtituting \eqref{EC1} to \eqref{EL1}, we obtain
\begin{eqnarray}\label{EL2}
    E_l[u]\le \bar{C}\frac{ \left( 1+t+t^2 \right)}{t^2}\left( {{E}_{\mathcal{C}}}(0)+{E}\left[ \slashed{\nabla }u \right]\left( 0 \right) +E(0)\right)~.
\end{eqnarray}
Thus, using \eqref{EL2} and \eqref{holderEl} to \eqref{Ec_acent2}, we have
\begin{eqnarray}
      {{E}_{\mathcal{C}}}(t') &\le& \bar{C}\left({{ E_{\mathcal{C}}(0)+ \left({E}\left[ \slashed{\Delta }^2u \right]\left( 0 \right) \right)}^{1/4}}{{\left( {{E}_{\mathcal{C}}}(0)+E\left[ \slashed{\nabla }u \right]\left( 0 \right)+E[u](0) \right)}^{3/4}}\right. \notag\\ 
 & \cdot& \left.\left\{ \underset{t\in \left[ 0,{t}' \right]}{\mathop{\sup }}\,\left( t{{\left( \frac{\left( 1+t+t^2 \right)}{{{t}^{2}}} \right)}^{3/4}} \right)+\int_{0}^{{{t}'}}{{{(1+t)\left( \frac{\left( 1+t+t^2\right)}{{{t}^{2}}} \right)}^{3/4}}}dt \right\}\right)~. \notag\\
\end{eqnarray}
Given that for $t\ge a\ge 1$, we have the inequality
\begin{eqnarray}
  {{\left( \frac{\left( 1+t+t^2\right)}{{{t}^{2}}} \right)}^{3/4}} \le C~,
\end{eqnarray}
and recognizing that $E[u](t) \leq E[\nabla u](t) \leq E[\Delta u](t)\leq E[\Delta^2 u](t)$, we can thus derive an upper bound for the conformal energy over any finite time interval $t$
\begin{eqnarray}\label{EC_t1/4}
{{E}_{\mathcal{C}}}(t) \le \bar{C}(1+t)^2\left( {{E}_{\mathcal{C}}}(0)+E\left[ \slashed{\Delta^2 }u \right]\left( 0 \right)+E[u](0) \right)~,t\ge a.
\end{eqnarray}
To get bound on $0\le t\le a$, one can use the exponential bound in \eqref{expEC},
\begin{eqnarray}
  E_{\mathcal{C}}(t)\le E_{\mathcal{C}}(0) e^{C\left(t^4+t^3\right)}  \le \bar{C} E_{\mathcal{C}}(0)
\end{eqnarray}
Thus for $t\ge 0$, we obtain \eqref{prep1}. This proves theorem \ref{statprep1}.
\end{proof}

\vskip0.5cm

\section*{Acknowledgements}
This work is supported by Posdoctoral Research Program BRIN 2024. B. E. G. and F. T. A. also acknowledge ITB Research Grant 2024 for financial support.


\vskip1cm
\begin{thebibliography}{20}
\bibitem{morawetz1961}
Morawetz C S 1961 {\it Commun. Pure Appl. Math} {\bf 14(3)} 561–568.
\vspace{-0.5em}
\bibitem{morawetz}
Morawetz C S 1975 {\it Commun. Pure Appl. Math} {\bf 28(2)} 229–264.
\vspace{-0.5em}

\bibitem{klainerman1980}
Klainerman S 1980 {\it Commun. Pure Appl. Math} {\bf33(1)} 43–101.
\vspace{-0.5em}
\bibitem{klainerman1}
Klainerman S 1985 {\it Commun. Pure Appl. Math} {\bf 38(3)} 321–332.
\vspace{-0.5em}

\bibitem{morawetz1962}
Morawetz C S 1962 {\it Commun. Pure Appl. Math} {\bf 15(3)} 349–361.
\vspace{-0.5em}
\bibitem{cho}
Choquet-Bruhat Y and Christodoulou D 1981 {\it Ann. Sci. Ecole. Norm. S} {\bf 14.4} pp. 481–506.
\vspace{-0.5em}

\bibitem{bruhat}
Choquet-Bruhat Y, Christodoulou D, and Francaviglia M 1979 {\it Annales de l'I.H.P. Physique théorique} {Vol \bf 31} no. 4 pp. 399-414.
\vspace{-0.5em}

\bibitem{klainerman1987}
Klainerman S 1987 {\it In: Sem. on in Non. Partl. Diff. Eq. Asp. of Math.} vol {\bf 10}. 
\vspace{-0.5em}

\bibitem{gini}
Ginibre J, Soffer A, and Velo G 1992 {\it J. Funct. Anal.} {\bf 110(1)} 96–130.
\vspace{-0.5em}
\bibitem{klainerman1994}
Klainerman S and  Machedon M 1994 {\it Duke Math. J}, {\bf 74(1)} 19–44.
\vspace{-0.5em}
\bibitem{klainerman1981}
Klainerman S and Sarnak P 1981 {\it Annales de l’I. H. P., section A}, tome {\bf 35} no 4 p. 253-257.
\vspace{-0.5em}

\bibitem{blue1}
Blue P and Soffer A 2003 {\it Adv. Differential Equations}{\bf 8(5)} 595-614.
\vspace{-0.5em}
\bibitem{blue}
Blue P 2008 {\it J. Hyperbolic Differ. Equ.} {\bf 05} pp. 807-856.
\vspace{-0.5em}

\bibitem{Finster}
Finster F and Smoller J 2008 {\it J. Hyperbolic Differ. Equ}{\bf 5(01)} 221–255.
\vspace{-0.5em}

\bibitem{dafermos}
Dafermos M and Rodnianski I 2009 {\it Commun. Pure Appl. Math} {\bf 62} 08590919.
\vspace{-0.5em}
\bibitem{Ander}
Andersson L, B\"ackdahl T, and Blue P 2016 {\it Class. Quantum Grav} {\bf 33} 085010.
\vspace{-0.5em}
\bibitem{mokdad} 
Mokdad M 2020 {\it Lett. Math. Phys} \textbf{110} 1961-2018. 
\vspace{-0.5em}
\bibitem{ghanem}
Ghanem S 2014 arXiv:{1409.8040}.
\vspace{-0.5em}


\bibitem{aubin}
Aubin T 1982  {\it Nonlinear Analysis on Manifolds} (New York: Springer-Verlag).



\end{thebibliography}
\end{document}